\documentclass[12pt, twoside]{article}
\usepackage{amsfonts}
\usepackage{xcolor}
\usepackage{mathrsfs}
\usepackage[all,cmtip]{xy}
\usepackage{amsmath}
\usepackage{amssymb,amsthm,upref,amscd}
\usepackage{setspace}
\usepackage{enumerate}
\usepackage[titletoc]{appendix}
\usepackage{times}
\usepackage{cite}
\usepackage{tikz}
\usepackage{authblk}
\usepackage[colorinlistoftodos]{todonotes}
\usetikzlibrary{arrows}
\numberwithin{equation}{section}

\def\titlerunning#1{\gdef\titrun{#1}}
\makeatletter
\def\author#1{\gdef\autrun{\def\and{\unskip, }#1}\gdef\@author{#1}}

\makeatother

\def\subjclass#1{{\renewcommand{\thefootnote}{}%
\footnote{\emph{Mathematics Subject Classification (2010):} #1}}}
\def\keywords#1{\par\medskip
\noindent\textbf{Keywords.} #1}

\allowdisplaybreaks

\theoremstyle{plain}
\newtheorem{Thm}{Theorem}[section]
\newtheorem{Lem}[Thm]{Lemma}
\newtheorem{claim}{Claim}[section]
\newtheorem{Cor}[Thm]{Corollary}
\newtheorem{Prop}[Thm]{Proposition}
\newtheorem*{Thm*}{Theorem}
\newtheorem*{claim*}{Claim}
\theoremstyle{definition}

\newtheorem*{Def*}{Definition}
\newtheorem*{Cor*}{Corollary}
\newtheorem{Rem}[Thm]{Remark}

\newtheorem{Ques}{Question}

\newcommand{\equ}{equation}
\newcommand{\C}{\mathbb{C}}
\newcommand{\N}{\mathbb{N}}
\newcommand{\R}{\mathbb{R}}
\newcommand{\Z}{\mathbb{Z}}

 \DeclareMathOperator{\dist}{dist}
 
\DeclareMathOperator{\meas}{meas}
\DeclareMathOperator{\supp}{supp}

\DeclareMathOperator{\vol}{vol}

\DeclareMathOperator{\spann}{span}

\newcommand\loc{\mathrm{loc}}
\newcommand\eps{\varepsilon}

\let\nhatoksa=\theenumi
\let\nhatoksb=\labelenumi
\let\nhatoksc=\theenumii
\let\nhatoksd=\labelenumii
\newlength{\nhalengtha}
\newlength{\nhalengthb}
\newlength{\nhalengthc}
\newcommand{\resetenum}{
\let\theenumi=\nhatoksa
\let\labelenumi=\nhatoksb
\let\theenumii=\nhatoksc
\let\labelenumii=\nhatoksd
\setlength{\leftmargini}{\nhalengtha}
\setlength{\leftmarginii}{\nhalengthb}
\setlength{\labelwidth}{\nhalengthc}
}

\newcommand\cc{\mathcal{C}}

\newcommand\cg{\mathcal{G}}
\newcommand\ch{\mathcal{H}}

\newcommand\ck{\mathcal{K}}
\newcommand\cl{\mathcal{L}}

\newcommand\co{\mathcal{O}}

\newcommand\cs{\mathcal{S}}

\newcommand\cw{\mathcal{W}}

\def\mbs{\mathbb{S}}

\def\msh{\mathscr{H}}

\def\msn{\mathscr{N}}

\def\ig{\textit{g}}

\def\ov{\overline}

\def\pa {\partial}

\def\op{\oplus}

\def\De{\Delta}

\def\ka {\kappa}
\def\al {\alpha}
\def\bt {\beta}
\def\de {\delta}
\def\Ga {\Gamma}
\def\ga {\gamma}
\def\lm {\lambda}
\def\Lam {\Lambda}
\def\om {\omega}
\def\Om {\Omega}
\def\sa {\sigma}
\def\vr {\varepsilon}
\def\va {\varphi}

\def\span{\hbox{span}}

\def\meas{\hbox{\it meas}}

\def\vol{\mathrm{vol}}
\def\real{\hbox{Re}}

\newcommand{\inp}[2]{\left\langle#1,#2\right\rangle}

\begin{document}

\baselineskip=17pt

\titlerunning{Conformal embeddings of $S^2\to\R^3$ with prescribed mean curvature: A variational approach}

\title{   \vspace{-3em}
Conformal embeddings of $S^2\to\R^3$ with prescribed mean curvature: A variational approach}

\author{Tian Xu\footnote{\noindent
 Supported by the National Science Foundation of China (NSFC 11601370, 11771325)
and the Alexander von
\newline \hspace*{1.5em}
Humboldt Foundation of Germany}  \vspace{-2em}
}

\date{}

\maketitle


\begin{abstract}

Motivated by recent progress on a spinorial analogue of the Yamabe problem in the geometric literature, we study a conformally invariant spinor field equation on the $m$-sphere, $m\geq2$. Via variational methods and the spinorial Weierstra\ss\ representation, we study the problem of prescribing mean curvature for the immersion $S^2\to\R^3$.


\end{abstract}


\section{Introduction}

One of the fundamental topics in classical differential geometry concerns the question of whether a smooth Riemannian manifold can be isometrically immersed/embedded into Euclidean spaces. It has been conjectured by Schl\"afli in 1873 that every $m$-dimensional smooth Riemannian manifold $(M^m,\ig)$ admits a local isometric embedding into $\R^{N_m}$, with $N_m=\frac{m(m+1)}{2}$. Several important achievements have been made towards this problem in the last century. In one of his outstanding papers, J. Nash \cite{Nash} showed the existence of a global isometric embedding of any $m$-dimensional smooth Riemannian manifold in $\R^{N}$ with $N=3N_m+4m$ in the case of compact $M$ and $N=(m+1)(3N_m+4m)$ in the case of non-compact $M$. Following Nash, we mention the book "Partial Differential Relations" published in 1986 by Gromov, which contains various problems related to the isometric embedding of Riemannian manifolds. Furthermore, in this book, Gromov showed that $N=N_m+2m+3$ is enough for the compact case.

In low dimensions, if one asks the more specific question of whether a given $2$-dimensional Riemann surface can be isometrically immersed/embedded into Euclidean $3$-space, not too much is known. Besides some general local existence results (see for example \cite{Spivak, Lin} for real analytic metrics and for smooth metrics under suitable curvature assumptions), non-existence results seem to be easier to come by: for instance, Hilbert's classical result that the hyperbolic plane does not admit an isometric immersion into $\R^3$, and the fact that a compact non-positively curved Riemann surface cannot be isometrically immersed into $\R^3$.

In the classical differential geometry of surfaces, the Gau\ss-Codazzi-Mainardi equations consist of a pair of related equations, and it incorporates the extrinsic curvature (or mean curvature) of the surface. The equations are precisely the integrability conditions to obtain an immersed surface in $\R^3$ (locally or on a simply connected domain).  With the development of Spin Geometry and the theory of Dirac operators, an equivalent albeit simpler expression becomes well-known, namely the existence of a special spinor field called {\it generalized Killing spinor field} (see for instance \cite{KS, Roth, Taimanov1, Taimanov2, Taimanov3} and references therein). This expression is
the so-called \textit{spinorial Weierstra\ss\ representation}, and can be understood as a geometrically invariant way for the representation of surfaces in Euclidean $3$-space.

In late 1990s, T. Friedrich has written a
very nice article in which the spinorial Weierstra\ss\
representation is beautifully explained,
see \cite[Theorem 13]{Friedrich1998}. Eventually, it turns out that there exists a one-to-one correspondence between the existence
of an isometric immersion $(\widetilde M^2,\ig)\hookrightarrow(\R^3,\ig_{\R^3})$
with mean curvature $H$ and the existence of a spinor field $\va$ with constant length satisfying the following Dirac equation
\begin{\equ}\label{linear Dirac}
D_\ig\va=H\va, \quad |\va|_\ig\equiv1
\end{\equ}
where $\widetilde M^2$ is the universal covering of $M^2$,  $D_\ig$
and $|\cdot|_\ig$ are respectively the Dirac operator and the  hermitian metric on the spinor bundle $\mbs(M^2)$ over $(M^2,\ig)$.

Following this idea, the situation is much clearer for simply connected surfaces. And one may naturally ask the following question:
\begin{Ques}\label{Q1}\vspace{-0.5em}
{\it Given an arbitrary smooth function $H: S^2\to\R$, what is the condition on $H$ so that there is an isometric immersion $S^2\hookrightarrow\R^3$ realizing $H$ as its mean curvature?}
\end{Ques}\vspace{-0.5em}

\noindent
Or it is equivalent to consider the following question via the spinorial Weierstra\ss\ representation:

\begin{Ques}\label{Q2}\vspace{-0.5em}
{\it What are necessary and sufficient conditions on $H$ so that there is a solution to Eq. \eqref{linear Dirac}?}
\end{Ques}

It seems interesting to see that Question \ref{Q2} can be investigated via a conformal change of Eq. \eqref{linear Dirac}. That is, we may begin with the standard Riemann sphere $(S^2,\ig_{S^2})$ and consider the following nonlinear
partial differential equation involving the Dirac operator $D_{\ig_{S^2}}$:
\begin{\equ}\label{Dirac2}
D_{\ig_{S^2}}\psi=H|\psi|_{\ig_{S^2}}^2\psi \quad \text{on }
(S^2,\ig_{S^2}).
\end{\equ}
Suppose we have a spinor field $\psi$ that satisfies Eq. \eqref{Dirac2}
and that never vanishes. Then one may
introduce a conformal metric $\ig=|\psi|_{\ig_{S^2}}^4\ig_{S^2}$
on $S^2$. Through the well-known formula for the Dirac operator under a conformal change of the metrics (see for instance \cite[Proposition 1.3.10]{Ginoux} and Proposition \ref{conformal formula} below), we obtain a new spinor field $\va:=\iota(|\psi|_{\ig_{S^2}}^{-1}\psi)$ which satisfies
\[
D_{\ig}\va=H\va, \quad |\va|_{\ig}\equiv1,
\]
where $\iota:(\mbs(S^2),\ig_{S^2})\to(\mbs(S^2),\ig)$
is a fiberwise isometry. And thus we find an
isometric immersion $(S^2,\ig)\hookrightarrow(\R^3,\ig_{\R^3})$ that
realizes $H$ as prescribed mean curvature.

In this paper, we will make use of the above spinorial tool in the study
of Question \ref{Q2}. From this point of view,
it seems natural to consider a two-step procedure:
\begin{itemize}
\item find a non-trivial spinor field $\psi$ satisfying
Eq. \eqref{Dirac2},

\item  show that the zero set under
$|\psi|_{\ig_{S^2}}: S^2\to[0,\infty)$ is empty, i.e. $\psi^{-1}(0)=\emptyset$.
\end{itemize}
However this is not an easy task. In general, even if one already knows there exists a non-trivial solution $\psi$ to Eq. \eqref{Dirac2}, it is still difficult to know whether $\psi$ possesses zeros. This is due to the lack of an adequate replacement for the maximum principle developed for second order elliptic partial differential operators.

Let us return to the problem of finding isometric immersions/embeddings of surfaces. The spinorial Weierstra\ss\
representation has transformed the problem into the equivalent problem which amount to solve a PDE on the spinor bundle over the surface. However, still, the picture is not completely clear for immersed surfaces. A further question besides Question \ref{Q1} and \ref{Q2} is the following:

\begin{Ques}\label{Q3}\vspace{-0.5em}
{\it Assume that a function $H:S^2\to\R$ is given such that it is the mean curvature of an immersion. Under which condition is this immersion an embedding? And when does it have self-intersections?}
\end{Ques}

In this paper, we are interested in answering Questions \ref{Q1}-\ref{Q3}. In order to explain our results in more detail, we begin with the following theorem, which provides a description of the zero set (or nodal set) of a solution to \eqref{Dirac2}.

\begin{Thm}\label{main theorem3}
On a closed spin surface $(M^2,\ig)$
of genus $\ga$, suppose that $H: M^2\to(0,\infty)$
is a smooth function. Let $\psi$ be a solution of the equation
\begin{\equ}\label{Dirac3}
D_\ig\psi=H|\psi|_{\ig}^2\psi \quad \text{on } M^2.
\end{\equ}
Then the number of zeros of $\psi$ is at most
\[
\ga-1+\frac{\int_{M^2}H^2|\psi|_\ig^4d\vol_{\ig}}{4\pi}.
\]
\end{Thm}
\begin{Rem}
We note that, by the regularity arguments in \cite{Isobe-JFA}
(see also \cite{Wang}),
a weak solution to Eq. \eqref{Dirac3}
is smooth provided that $H$ is smooth. Then it follows from
B\"ar's theorem (cf. \cite{Bar1997}) that
the zero set of $\psi$ should be discrete.
Theorem \ref{main theorem3} may be used to derive upper estimates for the number of nodal points. In order to explain this we write $H_{max}=\max_{M^2} H$ and we suppose that $\psi$ satisfies the energy estimate
\begin{\equ}\label{energy00}
\int_{M^2}H|\psi|_\ig^4d\vol_\ig<\frac{8\pi}{H_{max}}
\end{\equ}
and $M^2$ has genus $0$ (i.e. $M^2$ is topologically a sphere), then we have
\[
\#\psi^{-1}(0)\leq -1+\frac{\int_{M^2}H^2|\psi|_\ig^4d\vol_{\ig}}{4\pi}
\leq
-1+\frac{H_{max}\int_{M^2}H|\psi|_\ig^4d\vol_{\ig}}{4\pi}<1.
\]
And thus $\psi$ will have no zero at all.
\end{Rem}

After having studied the zero set, we discuss the existence of a (weak) solution to \eqref{Dirac3} on the $2$-sphere in the next theorem. In order to state the result, we introduce the following notation: For a function $H\in C^2(S^2)$, we set
$H_{max}:=\max_{\xi\in S^2}H(\xi)$ and $H_{min}:=\min_{\xi\in S^2}H(\xi)$. We then denote the set of maximum points as $\ch:=\big\{\xi\in S^2:\, H(\xi)=H_{max} \big\}$,
and $\ch_\de:=\big\{ \xi\in S^2:\, \dist_{\ig_{S^2}}(\xi,\ch)<\de \big\}$ its $\de$-neighborhood for $\de>0$.

Our criteria of the function $H$ will be formulated as
\begin{itemize}
\item[$(H)$] {\it $\ch$ is not contractible in its $\de$-neighborhood $\ch_\de$, for some small $\de>0$, but there exits $d\in\big(
    \max\big\{\frac12 H_{max}, \, H_{min}\big\},
    H_{max} \big)$  such that $\ch$ is contractible in
    $\big\{ \xi\in S^2:\, H(\xi)\geq d \big\}$.
    There is no critical value of $H$ in the interval $(d,H_{max})$, and
    if $\xi\in S^2$ is a critical point of $H$ with $H(\xi)=d$ then the Hessian of $H$ at $\xi$ is positive definite.}
\end{itemize}
And then, combining Theorem \ref{main theorem3}, the result is as follows

\begin{Thm}\label{main theorem1}
If $H: S^2\to(0,\infty)$ is a $C^2$ function satisfying condition $(H)$, then there is a solution $\psi$ to Eq. \eqref{Dirac2} with the energy estimate
\[
\frac{4\pi}{H_{max}}<\int_{S^2}H|\psi|_{\ig_{S^2}}^4d\vol_{\ig_{S^2}}<\frac{8\pi}{H_{max}}.
\]
Particularly, if $H$ is smooth, then the solution $\psi$ has no zero at all, i.e. the nodal set $\psi^{-1}(0)$ is empty.
\end{Thm}

\begin{Rem}\label{willmore remark}
By virtue of the spinorial Weierstra\ss\ representation, Theorem \ref{main theorem1} implies the existence of an isometric immersion $\Pi: (S^2,|\psi|_{\ig_{S^2}}^4\ig_{S^2})\to(\R^3,\ig_{\R^3})$.
In particular, the pull-back of the Euclidean volume form under this
immersion will be $\Pi^*(d\mu)=|\psi|_{\ig_{S^2}}^4 d\vol_{\ig_{S^2}}$. Notice that the Willmore functional for this immersion $\Pi$, denoted by $\cw(\Pi)$, appears as the integral squared norm of the mean curvature $H$. Hence, we have the estimate
\[
\cw(\Pi)=
\int_{S^2}H^2|\psi|_{\ig_{S^2}}^4 d\vol_{\ig_{S^2}}
\leq H_{max}\int_{S^2}H|\psi|_{\ig_{S^2}}^4 d\vol_{\ig_{S^2}}
<8\pi.
\]
Due to Li-Yau's inequality \cite[Theorem 6]{LY}, the immersion
$\Pi$ covers points in $\R^3$ at most once.
And thus $\Pi$ is in fact an isometric embedding.
\end{Rem}

\begin{figure}[ht]
\centering
\scalebox{0.5}[0.5]
 { \includegraphics{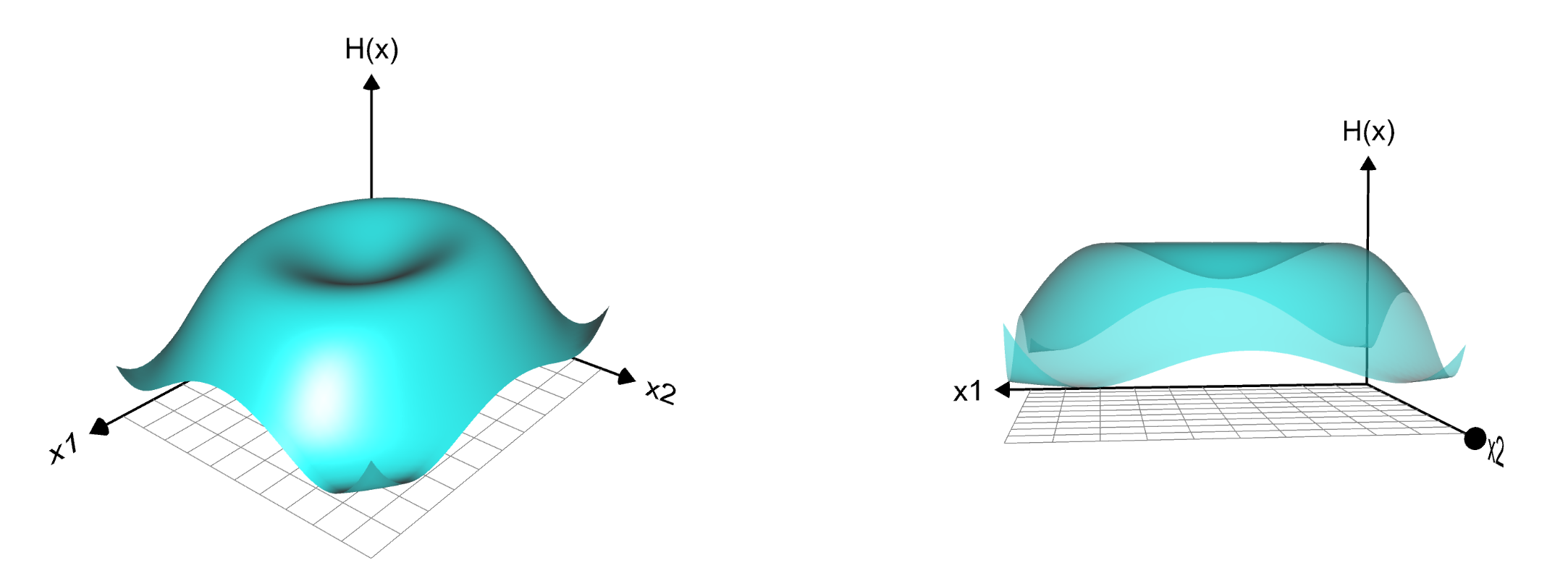}}
 \caption{A possible shape of function $H$'s, in local coordinates, which satisfies the criteria $(H)$.}
\end{figure}

We note that the class of functions $H$ that satisfy the hypothesis $(H)$ is dense, in $C^1$-topology, in the space of positive smooth functions. Then Theorem \ref{main theorem1} and Remark \ref{willmore remark} together immediately imply the next corollary,
which provides an answer to Question \ref{Q1}-\ref{Q3}. For ease of notation, we will denote
$[\ig_{S^2}]=\big\{
f^2\ig_{S^2}: \, f\in C^\infty(S^2), f>0\big\}$
the standard conformal class of $\ig_{S^2}$.

\begin{Cor}\label{main theorem2}
Let $C_+^\infty$ denote the class of positive smooth
functions on $S^2$ and let
\[
\msh=\big\{H\in C_+^\infty: \,
H\equiv constant \text{ or }
H \text{ satisfies condition } (H) \big\}.
\]
Then $\msh$ has the following properties:
\begin{itemize}
\item[$(1)$] $\msh$ is dense in $C_+^\infty$ with respect
to the $C^1$-topology;

\item[$(2)$] for any $H\in\msh$, there exists a
conformal metric $\ig_H\in[\ig_{S^2}]$ and an
isometric immersion
$\Pi_H:(S^2,\ig_H)\to(\R^3,\ig_{\R^3})$ realizing
$H$ as its mean curvature;

\item[$(3)$] the Willmore energy of $\Pi_H$ satisfies
\[
\cw(\Pi_H)=\int_{S^2}H^2d\vol_{\ig_H}<8\pi,
\]
and thus $\Pi_H$ is an isometric embedding.
\end{itemize}
\end{Cor}

We point out that, in \cite{Isobe-MathAnn}, T. Isobe also considered the Question \ref{Q1} and \ref{Q2}. He proved that when the function $H$ is very close to a constant, that is $H(\xi)=1+\vr Q(\xi)$ on $S^2$ with $Q$ being a specific Morse function and $\vr>0$ is small, then there exists a branched immersion $\Pi_{\vr,Q}:S^2\to\R^3$ that realizes its mean curvature as $1+\vr Q$. Here, by branched immersion we mean that $\Pi_{\vr,Q}$ is an immersion except at a discrete set of points. The method in \cite{Isobe-MathAnn} is also based on the spinorial Weierstra\ss\ representation and the nonlinear Dirac equation \eqref{Dirac2}.

Another result concerning this problem is due to M. Anderson \cite{Anderson}. He proved that for any positive function $H : S^2\to (0,\infty)$ in the class $C^{k,\al}$,
$k\geq0$ and $\al\in(0,1)$, there exists a branched immersion
$\Pi: S^2 \to \R^3$ and an positive affine function $\ell$ on $S^2$ such that the mean curvature of $\Pi$ is prescribed by $H+\ell$. In this setting the function $H>0$ lies arbitrarily in the class $C^{k,\al}$, and thus it can be far away from a constant. However, the affine function $\ell$ cannot be prescribed.

Note that, in Corollary \ref{main theorem2}, the mean curvature is accurately prescribed by the functions in the class $\msh$ and, particularly, the presence of branch points is ruled out. Hence, comparing with \cite{Isobe-MathAnn, Anderson}, our approach is in a different perspective. The class $\msh$ has its generality in the sense that only a "volcano"-shaped assumption is required and, furthermore, such a class of functions can give an explicit estimate of the Willmore functional from above. However, Corollary \ref{main theorem2} can not cover any of the results in \cite{Isobe-MathAnn, Anderson} but has intersection with them. There are many functions lying outside $\msh$ that can be prescribed as mean curvature of certain conformal isometric embeddings. However, there are still several problems that have to be solved.

Last but not least, we want to remind the readers that Eq. \eqref{Dirac2} for an arbitrarily given $H$ does not necessarily admit a non-trivial solution. As was shown in \cite{AHA}, there is an obstruction to the existence of a non-trivial solution to Eq. \eqref{Dirac2}. In fact, if there exists a non-trivial solution $\psi$ to Eq. \eqref{Dirac2}, then $H$ must satisfy
\begin{\equ}\label{obs}
\int_{S^2}(\nabla_X H)|\psi|_{\ig_{S^2}}^4d\vol_{\ig_{S^2}}=0
\end{\equ}
for any conformal vector field $X$ on $S^2$. In particular, $\nabla_X H$ must change sign on $S^2$. The obstruction \eqref{obs} is quite similar to the Kazdan-Warner obstruction for Gaussian and scalar curvature equations (see \cite{KW}).

Let us give an overview of our paper as follows.  We will consider the following semilinear Dirac equation with a critical exponent in all dimensions:
\begin{\equ}\label{Dirac m}
D_{\ig_{S^m}}\psi=H|\psi|_{\ig_{S^m}}^{\frac2{m-1}}\psi
\quad \text{on } (S^m,\ig_{S^m})
\end{\equ}
under the generalized assumption that $H\in C^2(S^m)$ satisfies
\begin{itemize}
\item[$(H*)$] {\it $\ch$ is not contractible in its $\de$-neighborhood $\ch_\de$, for some small $\de>0$, but there exits $d\in\big(
    \max\big\{2^{-\frac1{m-1}}H_{max}, \, H_{min}\big\},
    H_{max} \big)$  such that $\ch$ is contractible in
   $\big\{ \xi\in S^m:\, H(\xi)\geq d \big\}$.
    There is no critical value of $H$ in the interval $(d,H_{max})$,
    and if $\xi\in S^m$ is a critical point of $H$ with
    $H(\xi)=d$ then the Hessian of $H$ at $\xi$ is positive definite.}
\end{itemize}
Here, analogous to the $2$-sphere case, $H_{max}:=\max_{\xi\in S^m}H(\xi)$ and $H_{min}:=\min_{\xi\in S^m}H(\xi)$,  $\ch:=\big\{\xi\in S^m:\, H(\xi)=H_{max} \big\}$ is the collection of maximum points and $\ch_\de:=\big\{ \xi\in S^m:\, \dist_{\ig_{S^m}}(\xi,\ch)<\de \big\}$ is its $\de$-neighborhood for $\de>0$. Inspired by the idea of Yamabe \cite{Yamabe} and Aubin \cite{Aubin} when they deal with the Yamabe problem, here we mainly focus on a specific regularization of Eq. \eqref{Dirac m}, namely, a family of related equations with subcritical exponent
\begin{\equ}\label{Dirac p}
D_{\ig_{S^m}}\psi=H|\psi|_{\ig_{S^m}}^{p-2}\psi
\quad \text{on } (S^m,\ig_{S^m})
\end{\equ}
where $p\in(2,\frac{2m}{m-1})$. If one considers the variational formulation where solutions of Eq. \eqref{Dirac p} are obtained as critical points of the energy functional
\[
\cl_p(\psi):=\frac12\int_{S^m}(D\psi,\psi)d\vol_{\ig_{S^m}}
-\frac1p\int_{S^m} H|\psi|^{p}
d\vol_{\ig_{S^m}},
\]
then the main ingredient of this paper lies in the compactness of a family of critical points $\psi_p$ as $p\to\frac{2m}{m-1}$ of $\cl_p$. One way to show the compactness is to apply the regularity arguments to get uniform $C^{1,\al}$-boundedness for $\psi_p$, and hence by Arzel\`a-Ascoli theorem, up to a subsequence, the spinor field $\psi_p$ converges to a solution of Eq. \eqref{Dirac m}. This has been shown when the energy  $\cl_p(\psi_p)$ is below a so-called blow-up threshold (see \cite{Ammann, Raulot}). However, in our case, one must allow for blow-up phenomenon, and the situation becomes much more complicated. Whether or not one can obtain a uniform $C^{1,\al}$-boundedness for these $\psi_p$ is not clear. Here our strategy is to use build a concentration-compactness alternative for the sequence $\{\psi_p\}$. To rule out the blow-up phenomenon, one only needs to clarify that concentration will not occur, and this will be done in Section \ref{existence}.

From the view point of variational calculus, there are at least three different kinds of analytical difficulties when carrying out the above program. The first one is due to the strong indefiniteness of the functional $\cl_p$. It is well-known that the spectrum of the Dirac operator $D_{\ig_{S^m}}$ is neither bounded from above nor from below. Hence, a critical point of $\cl_p$ has infinite Morse index and infinite co-index. This will cause a problem since the classical variational methods such as the minimizing technique, Morse theory and its variants do not directly apply. The second difficulty is due to the existence of multiple solutions to the subcritical equation \eqref{Dirac p}. Note that Eq. \eqref{Dirac p} is invariant under the canonical action of $S^1=\{e^{i\theta}\in\C:\,\theta\in[0,2\pi]\}$ on spinors (i.e. if $\psi$ is a solution of Eq. \eqref{Dirac p} then $e^{i\theta}\psi$ is also a solution, for every fixed $\theta$). Moreover, for the case $m\equiv2,3,4$ (mod $8$), the spinor bundle has a quaternionic structure which commutes with Clifford multiplication, see for instance the construction in \cite[Section 1.7]{Friedrich} or \cite[Page 33, Table III]{Lawson}. Thus in these cases, Eq. \eqref{Dirac p} is invariant under the canonical action of the unit quaternions $S^3=\{q\in\mathbb{H}:\, |q|=1\}$ on spinors. And then one may be interested in the existence of $\cg$-inequivalent solutions of Eq. \eqref{Dirac p}, where $\cg=S^1$ or $S^3$. Standard variational arguments (see e.g. \cite[Chapter 2 Section 5]{Struwe}) show that the symmetric property of Eq. \eqref{Dirac p} under $\cg$-actions ensures multiple $\cg$-orbits of solutions. It can be seen from Remark \ref{ground state remark} that, among all these $\cg$-orbits of solutions, the ground state energy solutions should blow up when $p\to\frac{2m}{m-1}$. Whether or not there exists a sequence of solutions $\{\psi_p\}$ which shows compactness when $p$ approaches to $\frac{2m}{m-1}$ becomes a problem. Another difficulty lies in the analysis of the zero sets for a solution to Eq. \eqref{Dirac m}. Since the equation is defined on the spinor bundle, which is a complex vector bundle, in contrast to the Laplacian on functions, we do not have a maximum principle in order to exclude zeros for eigenspinors.

In what follows, the proof will be carried out as follows: in Section~\ref{preliminaries}, we introduce some notation and recall some known facts about the Dirac operator and blow-up asymptotic behavior of Eq. \eqref{Dirac p}; the analysis of zero sets for a solution to \eqref{Dirac m} on surfaces will be given in Section~\ref{nodal sec}; the detailed proof of Theorem \ref{main theorem1} and Corollary \ref{main theorem2} are given in Section~ \ref{existence}.


\section{Notation, definitions and known results}\label{preliminaries}

\subsection{The Dirac operator and $H^{\frac12}$-spinors}
In general, let $(M,\ig,\sa)$ be an $m$-dimensional compact spin manifold, where $\ig$ is a Riemannian metric on $M$, $\sa: P_{Spin}(M)\to P_{SO}(M)$ is a spin structure on $M$. We denote $\mbs(M)=P_{Spin}(M)\times_\rho\mbs_m$ the spinor bundle equipped with a natural hermitian inner product $(\cdot,\cdot)$. This is a complex vector bundle associated to the principle bundle $P_{Spin}(M)$ via the fundamental spin representation $\rho:Spin(m)\to End(\mbs_m)$. On the spinor bundle $\mbs(M)$, the Dirac operator
is then defined as the composition
\begin{displaymath}
\xymatrixcolsep{1.6pc}\xymatrix{
D_\ig:\Ga(\mbs(M)) \ar[r]^-{\nabla^\mbs} &
\Ga(T^*M\otimes \mbs(M)) \ar[r] &
\Ga(TM\otimes \mbs(M)) \ar[r]^-{\mathfrak{m}} &
\Ga(\mbs(M))}
\end{displaymath}
where $\nabla^\mbs$ is the canonical lift of the Levi-Civita connection on $P_{SO}(M)$ via the double covering $P_{Spin}(M)\to P_{SO}(M)$ and $\mathfrak{m}$ denotes the Clifford multiplication
$\mathfrak{m}: X\otimes\psi\mapsto X\cdot\psi$.

The Dirac operator behaves very nicely under conformal changes in the following sense, which was already known to Hitchin \cite{Hit74} and which was worked out in detail in \cite{Hij86}:
\begin{Prop}\label{conformal formula}
Let $\ig_0$ and $\ig=f^2\ig_0$ be two conformal metrics on
a Riemannian spin manifold $M$. Then, there exists an
isomorphism of vector bundles $\iota:\, \mbs(M,\ig_0)\to
\mbs(M,\ig)$ which is a fiberwise isometry such that
\[
D_\ig\big( \iota(\psi) \big)=\iota\big( f^{-\frac{m+1}2}D_{\ig_0}
\big( f^{\frac{m-1}2}\psi \big)\big),
\]
where $\mbs(M,\ig_0)$ and $\mbs(M,\ig)$ are spinor
bundles on $M$ with respect to the metrics $\ig_0$ and
$\ig$, respectively, and $D_{\ig_0}$ and $D_\ig$ are
the associated Dirac operators.
\end{Prop}

Let us consider the case $M=S^m$ with the standard metric $\ig_{S^m}$. For simplicity, we denote $D=D_{\ig_{S^m}}$ the associated Dirac operator on the spinor bundle $\mbs(S^m)$. Let $Spec(D)$ be its spectrum. It is well-known that $D$ is essentially self-adjoint in $L^2(S^m,\mbs(S^m))$ and has compact resolvents
(see \cite{Friedrich, Lawson, Ginoux}). Furthermore, we have
\[
Spec(D)=\Big\{\pm\big(\frac{m}2+j\big):\, j=0,1,2,\dots\Big\}
\]
and, for each $j$, the eigenvalues $\pm(\frac{m}2+j)$ have the
multiplicity
\[
2^{[\frac{m}2]}
\begin{pmatrix}
m+j-1\\
j
\end{pmatrix}
\]
(see \cite{Sulanke}). This concludes that the spectrum of the Dirac operator is discrete. For notation convenience, we can write
$Spec(D)=\{\lm_k\}_{k\in\Z\setminus\{0\}}$ with
\[
\lm_k=\frac{k}{|k|}\big(\frac{m}2+|k|-1\big).
\]
An application of the classical spectral theory implies that the eigenspaces of
$D$ form a complete orthonormal decomposition of
$L^{2}(S^m,\mbs(S^m))$, that is,
\[
L^{2}(S^m,\mbs(S^m))=\ov{\bigoplus_{\lm\in Spec(D)} \ker(D-\lm I)}.
\]

Let us denote $m(\lm_k)$ the multiplicity of $\lm_k\in Spec(D)$. It is clear that $L^{2}(S^m,\mbs(S^m))$ possesses a normalized Hilbert basis consisting of eigenspinors of $D$, that is, $\{\eta_{k,j}\}$ for $k\in \Z\setminus\{0\}$ and $1\leq j\leq m(\lm_k)$
with corresponding eigenvalues $\lm_k$. Hence $D\eta_{k,j}=\lm_k\eta_{k,j}$. We can then define the operator
$|D|^{1/2}: L^2(S^m,\mbs(S^m))\to L^2(S^m,\mbs(S^m))$ by
\[
|D|^{1/2}\psi=\sum_{k\in \Z\setminus\{0\}}\sum_{j=1}^{m(\lm_k)}|\lm_k|^{1/2}
a_{k,j}\eta_{k,j},
\]
for $\psi=\sum
a_{k,j}\eta_{k,j}\in L^2(S^m,\mbs(S^m))$.
Let us set
\[
H^{1/2}(S^m,\mbs(S^m)):=\bigg\{\psi=\sum
a_{k,j}\eta_{k,j}
\in L^2(S^m,\mbs(S^m)):\, \sum_{k\in \Z\setminus\{0\}}\sum_{j=1}^{m(\lm_k)}|\lm_k||a_{k,j}|^2<\infty\bigg\}. 
\]
The space $H^{1/2}(S^m,\mbs(S^m))$ coincides with the Sobolev space $W^{\frac12,2}(S^m,\mbs(S^m))$
(see \cite{Adams, Ammann}), and hence we have the Sobolev embeddings: $H^{1/2}(S^m,\mbs(S^m))\hookrightarrow L^q(S^m,\mbs(S^m))$ for $2\leq q\leq \frac{2m}{m-1}$.
In the sequel, we endow $E:=H^{1/2}(S^m,\mbs(S^m))$
with the inner product
\[
\inp{\psi}{\va}=\real
\big(|D|^{1/2}\psi,|D|^{1/2}\va\big)_2
\]
and the induced norm $\|\cdot\|
=\inp{\cdot}{\cdot}^{1/2}$,
where $(\psi,\va)_2=\int_{S^m}
(\psi,\va)d\vol_{\ig_{S^m}}$ is the
$L^2$-inner product on spinors. Let us mention here that $\|\cdot\|$ is equivalent to the graph norm for the operator $|D|^{1/2}$, and hence $(E,\|\cdot\|)$ is complete and $C^{\infty}(S^m,\mbs(S^m))$ is dense in $E$. Furthermore, $E$ induces a splitting $E=E^+\op  E^-$ with
\begin{\equ}\label{decomposition1}
E^+:=\ov{\span\{\eta_{k,j}\}_{k>0}} \quad
\text{and} \quad
E^-:=\ov{\span\{\eta_{k,j}\}_{k<0}},
\end{\equ}
where the closure is taken in the $\|\cdot\|$-topology.
It is then clear that these are orthogonal subspaces
of $E$ on which the action
$
\int_{S^m}(D\psi,\psi)d\vol_{\ig_{S^m}}
$
is positive or negative.
In the sequel, with respect to this decomposition,
we will write $\psi=\psi^++\psi^-$ for any
$\psi\in E$.  The dual space of $E$ will be denoted by
$E^*:=H^{-\frac12}(S^m,\mbs(S^m))$.

\subsection{Variational settings}\label{variational settings}

As was mentioned before,  we consider nonlinear Dirac equations of the following form:
\begin{\equ}\label{nld-p}
D\psi=H|\psi|^{p-2}\psi \quad \text{on } S^m,
\end{\equ}
where $H:S^m\to(0,+\infty)$ is a $C^2$ function, $p\in(2,2^*]$ and $2^*=\frac{2m}{m-1}$.

Eq. \eqref{nld-p} has a variational structure: $\psi$ is a solution to \eqref{nld-p} if and only if $\psi$ is a critical point of $\cl_p$ defined by
\[
\aligned
\cl_p(\psi)&=\frac12\int_{S^m}(D\psi,\psi)
d\vol_{\ig_{S^m}}-\frac1{p}
\int_{S^m} H|\psi|^{p}d\vol_{\ig_{S^m}}  \\
&=\frac12\big(\|\psi^+\|^2-\|\psi^-\|^2\big) - \frac1{p}
\int_{S^m} H|\psi|^{p}d\vol_{\ig_{S^m}}
\endaligned
\]
for $\psi=\psi^++\psi^-\in E=E^+\op E^-$.

In order to simplify the notation, we will occasionally use the notation $\cl$ instead of $\cl_{2^*}$ in the sequel. And for first and second derivatives, we write $\cl_p'(\psi)[\va]$ for the derivative of $\cl_p$ at $\psi$ applied to $\va$ and, similarly, we write $\cl_p''(\psi)[\va,\phi]$ for the second derivative of $\cl_p$ at $\psi$ applied to $\va$ and $\phi$. By $L^p$ we denote the Banach space $L^p(S^m,\mbs(S^m))$ for $p\geq1$ and by $|\cdot|_p$ we denote the usual $L^p$-norm.

Since both the negative and positive parts of $Spec(D)$ are unbounded, the quadratic form in $\cl_p$ is of strongly indefinite type.The next result can be viewed as a Lyapunov-Schmidt reduction of the strongly indefinite functional $\cl_p$. Specifically, it consists of a 2-step procedure: first to reduce the functional $\cl_p$ to the subspace $E^+$ and then to reduce it on a Nehari-Pankov manifold in $E^+$. This procedure will make the subsequent estimates more transparent. The idea of this reduction was already employed to the study of Choquard-Pekar equation in \cite[Lemma 2.1 \& 2.2]{BJS} and was also worked out in detail in a recent paper \cite{BX} for nonlinear Dirac equations on spin manifolds.

\begin{Prop}\label{reduction}
\begin{itemize}
	\item[$(1)$] There exists a $C^1$ map $h_p: E^+\to E^-$ such that for $\psi\in E$:
\begin{\equ}\label{hlm}
		\cl_p'(\psi)[v]=0 \quad \forall v\in E^- \quad \Longrightarrow  \quad \psi^-=h_p(\psi^+).
\end{\equ}
Moreover, $h_p(u)$ maximizes $\cl_p(u+v)$ over all $v\in E^-$ for $u\in E^+$.

\item[$(2)$] The functional $I_p: E^+\to\R$, $I_p(u)=\cl_p(u+h_p(u))$, satisfies
\[
I_p'(u)=0 \quad \Longrightarrow \quad  \cl_p'(u+h_p(u))=0.
\]

\item[$(3)$] For every $u\in E^+\setminus\{0\}$, the map $I_{p,u}:[0,\infty)\to\R$, $I_{p,u}(t)=I_p(tu)$, is of class $C^2$ and satisfies
\begin{\equ}\label{it}
I_{p,u}'(t)=0, \ t>0 \quad \Longrightarrow \quad  I_{p,u}''(t)<0.
\end{\equ}
Moreover $I_{p,u}(0)=I_{p,u}'(0)=0$ and $I_{p,u}''(0)>0$.
\end{itemize}
\end{Prop}

\begin{Rem} 
We omit the proof of Proposition \ref{reduction} since it can be done by following the arguments in \cite{BX}. Here some remarks are in order. The paper \cite{BX} has set up a general variational framework to study the equation $D\psi=\lm\psi+|\psi|^{2^*-2}\psi$ and its variations on an arbitrary spin manifold $M$, where $\lm\in\R$ is a real parameter. Proposition \ref{reduction} can be viewed as a counterpart of  \cite[Proposition 7.2]{BX} by simply taking $\lm=0$ and $M=S^m$. The appearance of the positive function $H$ and the subcritical exponent $p$ in the functional $\cl_p$ are rather harmless since they do not affect the strict convexity of the super-quadratic part. Hence there is no obstacle to apply the arguments in \cite{BX} and to obtain Proposition \ref{reduction}.   
\end{Rem}

In the sequel, we shall call $(h_p,I_p)$
the reduction couple for $\cl_p$ on $E^+$.
From Proposition \ref{reduction} $(1)$ and $(2)$, it follows that critical points of $I_p$ and $\cl_p$ are
in one-to-one correspondence via the injective map
$u\mapsto u+h_p(u)$ from $E^+$ to $E$.
The Nehari-Pankov manifold for $\cl_p$ is then defined as
\begin{\equ}\label{nehari}
	\msn_p=\big\{u\in E^+\setminus\{0\}:\,
	I_p'(u)[u]=0 \big\}
\end{\equ}
so that the reduced functional $I_p$ is bounded from below on $\msn_p$. In fact, since
$I_p(u)=\cl_p(u+h_p(u))$, we have 
\[
\aligned
I_p'(u)[w]&=\frac{d}{dt}I_p(u+tw)\Big|_{t=0}=\frac{d}{dt}\cl_p\big(u+tw+h_p(u+tw)\big)\Big|_{t=0} \\
&\hspace{-3em}
=\inp{u}{w}-\inp{h_p(u)}{h_p'(u)[w]}-\real\int_{S^m}H|u+h_p(u)|^{p-2}\big(u+h_p(u),w+h_p'(u)[w]\big)d\vol_{\ig_{S^m}}
\endaligned
\]
for any $w\in E^+$. Due to Proposition \ref{reduction} (1), we find $\cl_p'(u+h_p(u))[v]=0$ for all $v\in E^-$. This implies
\[
\aligned
0
=-\inp{h_p(u)}{v}-\real\int_{S^m}H|u+h_p(u)|^{p-2}\big(u+h_p(u),v\big)d\vol_{\ig_{S^m}}
\endaligned
\]
for all $v\in E^-$. By noting that the Fr\'echet derivative of $h_p$ at a point $u\in E^+$, which is denoted by $h_p'(u): E^+\to E^-$, is nothing but a linear operator. By taking $v=h_p'(u)[w]$ and $h_p(u)$ respectively in the above expression, we have 
\[
-\inp{h_p(u)}{h_p'(u)[w]}-\real\int_{S^m}H|u+h_p(u)|^{p-2}\big(u+h_p(u),h_p'(u)[w]\big)d\vol_{\ig_{S^m}}=0
\]
and
\begin{\equ}\label{X}
	-\|h_p(u)\|^2-\real\int_{S^m}H|u+h_p(u)|^{p-2}\big(u+h_p(u),h_p(u)\big)d\vol_{\ig_{S^m}}=0.
\end{\equ}
Therefore, the expression of $I_p'(u)[w]$ can be simplified as
\[
I_p'(u)[w]
=\inp{u}{w}-\real\int_{S^m}H|u+h_p(u)|^{p-2}\big(u+h_p(u),w\big)d\vol_{\ig_{S^m}}, \quad \text{for } w\in E^+
\]
and particularly, by taking \eqref{X} into account, $I_p'(u)[u]$ can be further rewritten as
\[
\aligned
I_p'(u)[u]&=\|u\|^2-\real\int_{S^m}H|u+h_p(u)|^{p-2}\big(u+h_p(u),u\big)d\vol_{\ig_{S^m}} \\
&=\|u\|^2-\|h_p(u)\|^2-\int_{S^m}H|u+h_p(u)|^{p-2}\big(u+h_p(u),u+h_p(u)\big)d\vol_{\ig_{S^m}} \\
&=\|u\|^2-\|h_p(u)\|^2-\int_{S^m}H|u+h_p(u)|^{p}d\vol_{\ig_{S^m}}.
\endaligned
\]
Now, for $u\in\msn_p$, we have $I_p'(u)[u]=0$ and
\begin{\equ}\label{Xx2}
	I_p(u)=I_p(u)-\frac12I_p'(u)[u]=\frac{p-2}{2p}\int_{S^m}H|u+h_p(u)|^{p}d\vol_{\ig_{S^m}}\geq0,
\end{\equ}
where the last inequality comes from the positiveness of $H$. Furthermore, we notice that $u\in\msn_p$ implies $u\in E^+\setminus\{0\}$ and $h_p(u)\in E^-$ (where $E^+$ and $E^-$ are orthogonal in $E$). Then $\{x\in S^m: \, \big|u+h_p(u)\big|(x)>0\}$ must have positive measure (otherwise, $u=-h_p(u)$ a.e. on $S^m$, which suggests that $u\in E^+\cap E^-=\{0\}$). Hence, \eqref{Xx2} can be slightly improved as
\[
I_p(u)=I_p(u)-\frac12I_p'(u)[u]=\frac{p-2}{2p}\int_{S^m}H|u+h_p(u)|^{p}d\vol_{\ig_{S^m}}>0.  
\]

By Proposition \ref{reduction} $(3)$, we have that $\msn_p$ is a $C^1$ manifold of codimension $1$ in $E^+$ and that $\msn_p$ is a natural constraint for the problem of finding non-trivial critical points of $I_p$ on $E^+$. Indeed, one may consider the auxiliary functional $\ck_p(u)=I_p'(u)[u]$ and can compute that $\ck_p'(u)[u]=I_p'(u)[u]+I_p''(u)[u,u]$. Then, for $u\in\msn_p$, one derives $\ck_p(u)=I_{p,u}'(1)=0$ and $\ck_p'(u)[u]=I_{p,u}''(1)<0$, which suggests $\msn_p$ to be a $C^1$ manifold. Moreover, according to the Lagrange multiplier rule, if $u\in\msn_p$ is a constrained critical point of $I_p$, then there exists $\lm\in\R$ such that
\[
I_p'(u)[v]=\lm\ck_p'(u)[v]=\lm \big(I_p'(u)[v]+ I_p''(u)[u,v]\big) \qquad \forall v\in E^+.
\]
Letting $v=u$, so that $I_{p,u}'(1)=I_p'(u)[u]=0$, we can conclude from the above relation and the fact $I_{p,u}''(1)=I_p''(u)[u,u]<0$ that we must have $\lm=0$, and hence $u$ is indeed a critical point of $I_p$ in $E^+$. We also mention that, for $u\in E^+\setminus\{0\}$, the function $I_{p,u}(t)$ attains its unique positive critical point at $t=  t(p,u)>0$ which realizes its maximum (correspondingly, $t(p,u)u\in\msn_p$ is called the projection of $u$ on $\msn_p$), and hence we have $\msn_p$ is homeomorphic to the unit sphere of $E^+$ and
\[ 
I_p(u)=\max_{t>0}I_p(tu)=\max_{\psi\,\in\, \spann\{u\}\op E^-}\cl_p(\psi) \quad \forall u\in \msn_p.
\]

In the sequel, it is convenient to consider the reduced functional $I_p$ restricted to the constraint $\msn_p$. However, an additional difficulty may arise in the variational arguments: it is nearly impossible to calculate the value of $I_p(u)$ at any point $u\in\msn_p$. Even if we choose an explicit spinor $\psi\in E\setminus\{0\}$, we cannot directly calculate its projection $\psi^+$ in $E^+$ and the corresponding projection $t(p,\psi^+)\psi^+$ in $\msn_p$, moreover, we know almost nothing about the images of $h_p(t\psi^+)$ for $t>0$. We state, in the next lemma, a result about an upper bound estimate for the reduced functional $I_p$ on $\msn_p$ in the directions of ``almost critical points" of $I_p$ that will play a significant role our argument.

\begin{Lem}\label{key2}
Let $p_n\nearrow 2^*$ be a converging sequence and suppose that there exists a sequence of spinor fields $\{\psi_n\}\subset E$ such that
\begin{\equ}\label{contradiction-assumption}
c_1\leq\cl_{p_n}(\psi_n)\leq c_2 \quad \text{and} \quad
\cl_{p_n}'(\psi_n)\to0
\end{\equ}
for some constants $c_1,c_2>0$, as $n\to\infty$. Then $\{\psi_n\}$ is bounded in $E$ such that
\[
\|\psi_n^--h_{p_n}(\psi_n^+)\|=O\big( \|\cl_{p_n}'(\psi_n)\|_{E^*} \big) 
\]
and
\[
\max_{t>0}I_{p_n}(t\psi_n^+)\leq
\cl_{p_n}(\psi_n)+O\big(\|\cl_{p_n}'(\psi_n)\|_{E^*}^2\big).
\]
\end{Lem}

Lemma \ref{key2} is obtained by collecting the results  in \cite[Lemma~4.1, Lemma~7.3 - Corollary~7.6]{BX} with some minor modifications. The idea in the proof is built upon the strict concavity of $I_{p,u}(t)$ around its maximum point (i.e. \eqref{it}) and an application of the Inverse Function Theorem to the auxiliary functional $\ck_{p_n}(u)=I_{p_n}'(u)[u]$ near the point $\psi_n^+$. As an immediate consequence of Lemma \ref{key2}, we have

\begin{Cor}\label{continuity prop}
Let $p_n\nearrow 2^*$ as $n\to\infty$ and $c_1,c_2>0$.
For any $\theta>0$, there exists $\al>0$ such
that for all large $n$ and $\psi\in E$ satisfying
\[
c_1\leq\cl_{p_n}(\psi)\leq c_2 \quad \text{and} \quad
\|\cl_{p_n}'(\psi)\|_{E^*}\leq\al
\]
there holds
\[
\max_{t>0}I_{p_n}(t\psi^+)\leq \cl_{p_n}(\psi)+\theta.
\]
\end{Cor}

For $p\in(2,2^*]$, let us set
\begin{\equ}\label{tau p def}
\tau_p:=\inf\big\{I_p(u):\, u\in\msn_p  \big\}.
\end{\equ}
As a consequence of the following estimate
\[
I_p(u)\geq \cl_p(u)=\frac12\|u\|^2-\frac1p\int_{S^m}H|u|^pd\vol_{\ig_{S^m}}\geq\frac12\|u\|^2- C\|u\|^p, \quad \forall u\in E^+
\] 
where $C>0$ is a positive constant depending on $H_{max}$ and the Sobolev embedding $E\hookrightarrow L^p$, it turns out immediately that $\tau_p$ is lower bounded by a positive constant.

\begin{Lem}\label{tau p properties}
Let $p_n\nearrow 2^*$ as $n\to\infty$, then $\tau_{p_n}\to \tau_{2^*}$ as $n\to\infty$.
\end{Lem}
\begin{proof}
For arbitrary $\eps>0$ and $k\in\N$, by the Ekeland variational principle \cite{Ekeland, MW, Willem}, we can first choose $u\in\msn_{2^*}$ such that $\tau_{2^*}\leq I_{2^*}(u)<\tau_{2^*}+\frac1k$ and $\|I_{2^*}'(u)\|\leq \frac1k$.

Let $s_n>0$ be such that $s_nu\in\msn_{p_n}$. It follows from the arguments in \cite[see page 182]{BJS} that $\{s_n\}$ must be bounded (otherwise, $I_{p_n}(s_nu)\to-\infty$ as $s_n\to+\infty$). 
And hence $\tau_{p_n}\leq I_{p_n}(s_nu)$ is bounded from above.

Note that $\|I_{2^*}'(u)\|=\|\cl_{2^*}'(u+h_{2^*}(u))\|$ (see for instance \cite[Lemma 2.2]{BJS}) and that
\[
\aligned
&\cl_{2^*}'(u+h_{2^*}(u))[\va]-\cl_{p_n}'(u+h_{2^*}(u))[\va] \\
&\qquad =\real\int_{S^m}H\big(1-|u+h_{2^*}(u)|^{2^*-p_n}\big)|u+h_{2^*}(u)|^{p_n-2}(u+h_{2^*}(u),\va)d\vol_{\ig_{S^m}} \\
&\qquad = o_n(1)\|\va\|
\endaligned
\]
for all $\va\in E$, we derive that $\|\cl_{p_n}'(u+h_{2^*}(u))\|\leq \frac2k$ for all large $n$. And thus, by applying Corollary \ref{continuity prop}, we have
\begin{\equ}\label{X1}
\tau_{p_n}\leq I_{p_n}(s_nu)=\max_{s>0}I_{p_n}(su)\leq \cl_{p_n}(u+h_{2^*}(u))+\eps 
\end{\equ}
provided that $k$ is fixed large enough. Since the function $p\mapsto \int_{S^m}H|u+h_{2^*}(u)|^pd\vol_{\ig_{S^m}}$ is continuous, we have that $\cl_{p_n}(u+h_{2^*}(u))=\cl_{2^*}(u+h_{2^*}(u))+o_n(1)=I_{2^*}(u)+o_n(1)$. Therefore, by \eqref{X1} and the arbitrariness of $\eps$, we have $\limsup_{n\to\infty}\tau_{p_n}\leq \tau_{2^*}$.

On the other hand, we can choose $u_n\in\msn_{p_n}$ such that $I_{p_n}(u_n)\leq \tau_{p_n}+\frac1n$. Let $t_n>0$ be such that $t_nu_n\in\msn_{2^*}$. Then it follows that $\{t_n\}$ is also bounded. And we can derive that
\begin{\equ}\label{X2}
\tau_{2^*}\leq I_{2^*}(t_nu_n)=\cl_{2^*}(t_nu_n+h_{2^*}(t_nu_n))\leq\cl_{p_n}(t_nu_n+h_{2^*}(t_nu_n))+o_n(1),
\end{\equ}
where in the last inequality we have used the fact
\[
\aligned
\int_{S^m} H|\psi|^{p_n}d\vol_{\ig_{S^m}} &\leq
\Big( \int_{S^m}H|\psi|^{2^*}d\vol_{\ig_{S^m}} \Big)^{\frac{p_n}{2^*}}\Big( \int_{S^m}Hd\vol_{\ig_{S^m}} \Big)^{\frac{2^*-p_n}{2^*}} \\
&=\int_{S^m}H|\psi|^{2^*}d\vol_{\ig_{S^m}}+o_n(1)
\endaligned
\]
when $\|\psi\|$ is bounded independent of $n$. Hence, we see from \eqref{X2} that $\tau_{2^*}\leq I_{p_n}(u_n)+o_n(1)\leq \tau_{p_n}+\frac1n+o_n(1)$ as $n\to\infty$, which completes the proof.
\end{proof}

We conclude this subsection by formulating the exact value of $\tau_{2^*}$, that we will need in the sequel, i.e.
	\begin{\equ}\label{tau value}
		\tau_{2^*}=\frac{1}{2m (H_{max})^{m-1}}
		\big(\frac m2\big)^m \om_m
	\end{\equ}
	where $\om_m$ stands for the volume of $(S^m,\ig_{S^m})$. To see this, we first need to show that $\tau_{2^*}\geq\frac{1}{2m (H_{max})^{m-1}}(\frac m2)^m \om_m$. It can be seen from the fact $H(\xi)\leq H_{max}$ for $\xi\in S^m$ that 
	\begin{\equ}\label{cl-max functional}
		\cl_{2^*}(\psi)\geq\cl_{max}(\psi):=\frac12\big(\|\psi^+\|^2-\|\psi^-\|^2\big) -\frac{H_{max}}{2^*}\int_{S^m}|\psi|^{2^*}d\vol_{\ig_{S^m}} 
	\end{\equ}
	for all $\psi\in E$. And critical points of $\cl_{max}$ correspond to solutions of
	\begin{\equ}\label{nld-max}
		D\psi=H_{max}|\psi|^{2^*-2}\psi \quad \text{on } S^m.
	\end{\equ}
	Since $H_{max}>0$, by normalizing Eq. \eqref{nld-max}, we turn to consider 
	\begin{\equ}\label{nld-normalized}
		D\psi=|\psi|^{2^*-2}\psi \quad \text{on } S^m.
	\end{\equ}
	Then solutions of Eq. \eqref{nld-max} and \eqref{nld-normalized} are in one-to-one correspondence via the map 
	\begin{\equ}\label{solu-to-solu}
		\psi\mapsto (H_{max})^{\frac{m-1}2}\psi.
	\end{\equ}
	Recall that $\R^m$ and $S^m\setminus\{N\}$ (where $N\in S^m$ is the North pole) are conformally equivalent, it follows from \cite[Proposition 4.1]{Isobe-JFA} and \eqref{solu-to-solu} that if $\psi$ is a nontrivial solution to Eq. \eqref{nld-max} then
	\begin{\equ}\label{max lower bound}
		\cl_{max}(\psi)\geq\frac{1}{2m (H_{max})^{m-1}}(\frac m2)^m \om_m.
	\end{\equ}
    Since it is well-known that Eq. \eqref{nld-normalized} has a specific solution, which is a Killing spinor to the constant $-1/2$ (see e.g. \cite[Section 5.3]{AGHM}), we can see from \eqref{solu-to-solu} that the equality in \eqref{max lower bound} is achieved. Furthermore, notice that Proposition \ref{reduction} can be applied to the functional $\cl_{max}$. Let us denote $I_{max}$ the reduced functional for $\cl_{max}$ on $E^+$ and $\msn_{max}$ the associated Nehari-Pankov manifold. Then, by \eqref{cl-max functional}, we have $I_{2^*}(u)\geq I_{max}(u)$ for all $u\in E^+$ and, particularly,
	\[
	\max_{t>0}I_{2^*}(tu)\geq I_{max}(u) \quad \text{for all } u\in\msn_{max}.
	\]
	Jointly with the definition of $\tau_{2^*}$ and \eqref{max lower bound}, we obtain
	\[
	\tau_{2^*}\geq\frac{1}{2m (H_{max})^{m-1}}(\frac m2)^m \om_m.
	\]
	The other inequality requires a delicate upper bound estimate for $\tau_{2^*}$. In fact, by a careful choice of the parameters, Lemma \ref{energy1} in Section \ref{existence} provides a proof, and one may also refer to a similar result in \cite{AGHM}.

%


\subsection{A concentration-compactness alternative}
Again, we choose $\{p_n\}$ to be a strictly increasing
sequence such that $\displaystyle\lim_{n\to\infty}p_n=2^*$.
In what follows, we shall collect some asymptotic properties of solutions to the equation
\begin{\equ}\label{equs-n}
D\psi=H|\psi|^{p_n-2}\psi \quad \text{on } S^m
\end{\equ}
with a specific energy constraints, the proofs will be postponed in the Appendix. First of all, an alternative result comes as follows:

\begin{Prop}\label{alternative prop}
Suppose $\{\psi_n\}\subset E$ is a sequence such that
\begin{\equ}\label{key-assumption}
\tau_{2^*}\leq\cl_{p_n}(\psi_n)
\leq2\tau_{2^*}-\theta \quad \text{and} \quad
\cl_{p_n}'(\psi_n)\to0, \quad
\text{as } n\to\infty
\end{\equ}
for some constant $\theta>0$. Then, up to a subsequence,
either $\psi_n\rightharpoonup 0$ in $E$ or
$\psi_n$ converges in $E$ to a non-trivial solution $\psi_0$ of Eq. \eqref{Dirac m}
\end{Prop}

In order to characterize the non-compact case (which corresponds to the weakly convergence $\psi_n\rightharpoonup 0$ in $E$) in the above alternatives, let us introduce the following notations. Let $Crit[H]=\{\xi\in S^m:\, \nabla H(\xi)=0\}$ denote the critical set of the function $H$.
Let $\xi\in V\subset S^m$ be an arbitrary point, we denote $\exp_\xi: U\subset T_\xi S^m\cong\R^m\to V$, $x=(x_1,\dots, x_n)\mapsto y=\exp_\xi x$ the exponential map and $(x_1,\dots,x_m)$ means the normal coordinates. For a sequence of points $\{a_n\}\subset S^m$ and an arbitrary decreasing sequence $R_n\searrow0$, as $n\to\infty$, we define the rescaled geodesic normal coordinates near each $a_n$ via $\mu_n(x)=\exp_{a_n}(R_n x)$.
Denoting $B_R^0=\big\{ x\in\R^m:\, |x|<R\big\}$, where
$|\cdot|$ is the Euclidean norm in $\R^m$, we have a
conformal equivalence $(B_R^0, \, R_n^{-2}\mu_n^*\ig_{S^m})
\cong (B_{R_nR}(a_n),\, \ig_{S^m})\subset S^m$ for all
large $n$. For ease of notation, we set $\ig_n=R_n^{-2}\mu_n^*\ig_{S^m}$.
By using the trivialization of the spinor bundle constructed by
Bourguignon and Gauduchon \cite{BG}, we see that the coordinate map $\mu_n$ induces a spinor bundle identification
 $\ov{(\mu_n)}_*:\mbs_x(B_R^0,\ig_n)\to
\mbs_{\mu_n(x)}(B_{R_nR}(a_n),\ig_{S^m})$ for each $n$. Moreover, due to the compactness of $S^m$, we may assume without loss of generality that $a_n\to a\in S^m$ as $n\to\infty$. Then, for all $n$ large, the aforementioned identifications depend smoothly in $a_n$. Particularly, we have the following blow-up result (the proof is postponed in the Appendix).

\begin{Prop}\label{blow-up prop}
Let $\{\psi_n\}\subset E$ satisfy \eqref{key-assumption}. If $\{\psi_n\}$
does not contain any compact subsequence.
Then, up to a subsequence
if necessary, there exist a sequence
$\{a_n\}\subset S^m$, $a_n\to a\in S^m$, as $n\to\infty$, a
sequence of radius $R_n\searrow0$, a real number $\lm\in(2^{-\frac1{m-1}},1]$ and a non-trivial solution $\phi_0$ of the equation
\[
D_{\ig_{\R^m}}\phi=\lm H(a)|\phi|^{2^*-2}\phi
\quad \text{on } \R^m
\]
 such that
\[
\lim_{n\to\infty}R_n^{\frac{m-1}2(2^*-p_n)}=\lm
\]
and
\[
\psi_n=R_n^{-\frac{m-1}2}\eta(\cdot)\ov{(\mu_n)}_*
\circ\phi_0\circ \mu_n^{-1}+o_n(1) \quad \text{in } E
\]
as $n\to\infty$,
where $\eta\in C^\infty(S^m)$ is a cut-off function
such that $\eta\equiv1$ on $B_r(a)$ and
$\supp\eta\subset B_{2r}(a)$, some $r>0$. Moreover, we
have
\[
\cl_{p_n}(\psi_n)\geq\frac{1}{2m (\lm H(a))^{m-1}}
\big(\frac m2\big)^m \om_m+o_n(1), \quad \text{as } n\to\infty
\]
where $\om_m$ stands for the volume of $(S^m,\ig_{S^m})$. If  $\cl_{p_n}'(\psi_n)\equiv0$ for all $n$ in \eqref{key-assumption}, then $a\in  Crit[H]$ and $\lm=1$.
\end{Prop}

\begin{Rem}\label{ground state remark}
\begin{itemize}
\item[(1)] Propositions \ref{alternative prop} and \ref{blow-up prop} are closely linked. More precisely, together with \eqref{tau value}, we can conclude  that, if $\{\psi_n\}$ is a sequence of critical points satisfying the energy estimates in \eqref{key-assumption} and does not contain any compact subsequence, then (up to a subsequence) the associated blow-up point $a\in Crit[H]$ satisfies $H(a)\geq2^{-\frac1{m-1}}H_{max}+\theta'$, for some $\theta'>0$.

\item[(2)] It is worth pointing out that, for $H\not\equiv constant$, the ground state energy solutions of \eqref{equs-n} (that is, the solutions with energy $\tau_{p_n}$) should blow up when $n\to\infty$. To see this, let us consider generically $\psi_n\in E$ such that $\cl_{p_n}(\psi_n)\to \tau_{2^*}$ and $\cl_{p_n}'(\psi_n)\to0$ as $n\to\infty$. By Proposition \ref{alternative prop}, it follows that either (a)  $\{\psi_n\}$ converges strongly to $\psi_0\neq0$ which satisfies Eq. \eqref{Dirac m} or (b) the sequence $\{\psi_n\}$ blows-up.

If (a) holds, then $\cl_{2^*}(\psi_0)=\tau_{2^*}$. To rule out this case, let us first recall the estimate \eqref{cl-max functional}, where the functional $\cl_{max}$ is introduced. As before, we denote $(h_{max},I_{max})$ the reduction couple for $\cl_{max}$ on $E^+$ and $\msn_{max}$ the associated Nehari-Pankov manifold. Then \eqref{cl-max functional} implies that $I_{2^*}(u)\geq I_{max}(u)$ for all $u\in E^+$. Since $\cl_{2^*}(\psi_0)=\tau_{2^*}$ and $\cl_{2^*}'(\psi_0)=0$, we have $\psi_0^+\in\msn_{2^*}$, $I_{2^*}(\psi_0^+)=\tau_{2^*}$ and $I_{2^*}'(\psi_0^+)=0$. Let $t_0>0$ be such that $t_0\psi_0^+\in\msn_{max}$. We claim that $t_0\neq1$. If, by contradiction, we have $t_0=1$. Then $\psi_0^+\in\msn_{2^*}\cap\msn_{max}$. By \eqref{max lower bound}, we find $I_{max}(\psi_0^+)=\tau_{2^*}$ and $I_{max}'(\psi_0^+)=0$ (this is because $\tau_{2^*}$ equals to the lowest energy level for  $I_{max}$ on the constraint $\msn_{max}$). Hence, we have
\begin{\equ}\label{XXX1}
	\tau_{2^*}=\frac12\big( \|\psi_0^+\|^2-\|h_{max}(\psi_0^+)\|^2 \big)-\frac{H_{max}}{2^*}\int_{S^m}|\psi_0^++h_{max}(\psi_0^+)|^{2^*}d\vol_{\ig_{S^m}}
\end{\equ}
and $\widetilde\psi_0:=\psi_0^++h_{max}(\psi_0^+)$ is a nontrivial solution of Eq. \eqref{nld-max}. Then, by the regularity result for weak solutions to Eq. \eqref{nld-max} (see e.g. \cite{CJW, Isobe-JFA, Wang}) and the {\it Weak Unique Continuation Property} for Dirac type operators (see \cite[Theorem 2.1]{BBMW}), we find that the zero set of $\widetilde\psi_0$ does not contain any nonempty open set. And therefore it follows directly from \eqref{XXX1} that
\[
\aligned
\tau_{2^*}&=\frac12\big( \|\psi_0^+\|^2-\|h_{max}(\psi_0^+)\|^2 \big)-\frac{H_{max}}{2^*}\int_{S^m}|\psi_0^++h_{max}(\psi_0^+)|^{2^*}d\vol_{\ig_{S^m}}\\
&<\frac12\big( \|\psi_0^+\|^2-\|h_{max}(\psi_0^+)\|^2 \big)-\frac{1}{2^*}\int_{S^m}H|\psi_0^++h_{max}(\psi_0^+)|^{2^*}d\vol_{\ig_{S^m}}  \\
&=\cl_{2^*}(\widetilde\psi_0) \leq I_{2^*}(\psi_0^+)=\tau_{2^*}
\endaligned
\]
which is a contradiction, proving the claim. Let's say $t_0\neq1$. Then we can infer
\[
\aligned
\tau_{2^*}&=I_{2^*}(\psi_0^+)>I_{2^*}(t_0\psi_0^+)\geq I_{max}(t_0\psi_0^+)\\
&\geq \inf_{u\in\msn_{max}}I_{\max}(u)=\frac{1}{2m (H_{max})^{m-1}}(\frac m2)^m \om_m=\tau_{2^*},
\endaligned
\]
which is again a contradiction. Thus we see that only (b) can occur.
\end{itemize}
\end{Rem}

\section{Analysis on the nodal set}\label{nodal sec}

In what follows, we shall study the nodal set of solutions to the following nonlinear Dirac equation with critical exponent
\begin{\equ}\label{eee}
D\psi=H|\psi|^{2^*-2}\psi \quad \text{on } S^m
\end{\equ}
And let us begin with
the following lemma which can be viewed as a
generalization of \cite[Lemma 4.1]{Ammann2009}
for the case $H\not\equiv constant$.
\begin{Lem}\label{Scal lemma}
Let $(M,\ig)$ be an arbitrary Riemannian spin $m$-manifold
(not necessary complete or compact). Assume that there
is a spinor $\psi$ of constant length $1$ and with
$D_\ig\psi=H\psi$ for a real-valued
smooth function $H:M\to\R$,
and let $\{e_1, \dots , e_m\}$ be a local orthonormal
frame field. Then
\[
Scal_\ig=\frac{4(m-1)}m
 H^2-4\sum_{k=1}^m|\nabla_{e_k}^H\psi|^2
\]
where $Scal_\ig$ is the scalar curvature of $M$ with respect
to $\ig$,
$\nabla^H_X\psi:=\nabla_X\psi+\frac{H}mX\cdot_\ig\psi$
for $X\in \Ga(TM)$ is an induced covariant derivative
and $\cdot_\ig$ denotes the Clifford multiplication.
\end{Lem}
\begin{proof}
The algebraic properties of Clifford multiplication
imply that $\nabla^H$ is metric in the sense that
\[
X(\va_1,\va_2)=(\nabla_X^H\va_1,\va_2)
+(\va_1,\nabla_X^H\va_2)
\]
for all $\va_1,\va_2\in\Ga(\mbs(M))$ and $X\in\Ga(TM)$.

Locally, let
$-\De^H=-\sum_{k=1}^m\nabla_{e_k}^H\nabla_{e_k}^H$
denote the Laplace operator corresponds to $\nabla^H$.
Then we have
\begin{eqnarray}\label{x5}
-\real(\De^H\psi,\psi)&=&-\real\sum_{k=1}^m
(\nabla_{e_k}^H\nabla_{e_k}^H\psi,\psi)  \nonumber\\
&=&-\real\sum_{k=1}^m\big[ e_k(\nabla_{e_k}^H\psi,\psi)
-(\nabla_{e_k}^H\psi,\nabla_{e_k}^H\psi) \big]
\nonumber\\
&=&-\real\sum_{k=1}^m e_k(\nabla_{e_k}^H\psi,\psi)
+\sum_{k=1}^m|\nabla_{e_k}^H\psi|^2.
\end{eqnarray}
Since $(\nabla_{e_k}^H\psi,\psi)=(\nabla_{e_k}\psi,\psi)
+\frac{H}m\cdot(e_k\cdot_\ig\psi,\psi)$, we have
\[
\real(\nabla_{e_k}^H\psi,\psi)=\real(\nabla_{e_k}\psi,\psi)
=\frac12 \real \inp{\text{grad}\,|\psi|^2}{e_k}=0.
\]
Thus \eqref{x5} gives us
\begin{\equ}\label{x6}
-\real(\De^H\psi,\psi)=
\sum_{k=1}^m|\nabla_{e_k}^H\psi|^2.
\end{\equ}

On the other hand, by the twisted version of the
Schr\"odinger-Lichnerowicz formula
(see for instance \cite[Chapter 5]{Friedrich}), we have
\[
\Big(D-\frac Hm\Big)^2=
-\De^H+\frac{Scal_\ig}4+\frac{1-m}{m^2}H^2.
\]
And therefore, by $D\psi=H\psi$, $|\psi|\equiv1$
and \eqref{x6}, we obtain
\begin{eqnarray*}
\frac{(m-1)^2H^2}{m^2}&=&-\real(\De^H\psi,\psi)
+\frac{Scal_\ig}4+\frac{1-m}{m^2}H^2  \\
&=&\sum_{k=1}^m|\nabla_{e_k}^H\psi|^2
+\frac{Scal_\ig}4+\frac{1-m}{m^2}H^2,
\end{eqnarray*}
which completes the proof.
\end{proof}

The following theorem is due to B\"ar \cite{Bar1997}:

\begin{Thm}\label{Bar}
Let $(M,\ig)$ be compact and connected spin $m$-manifold
and let $\va$ be a solution of
\[
D\va=P\va
\]
where $P$ is a smooth endomorphism. Then the zero set
of $\va$ has at most Hausdorff dimension $m-2$.
And if $m=2$, then the zero set is discrete.
\end{Thm}

Unfortunately, in general, Eq \eqref{eee} does not satisfy
B\"ar's theorem because $2^*-2=\frac2{m-1}\not\in\Z$ for
$m\geq4$ and $H|\psi|^{2^*-2}$ could not be smooth even for $m=3$.
However, in dimension $2$, we have
$2^*=4$ and in this case solutions of Eq \eqref{eee} and
the corresponding $P=H|\psi|^2$ are smooth enough provided that $H\in C^\infty$.
Thus, applying Theorem \ref{Bar}, we have the
nodal set of a solution $\psi$ is a discrete subset.

\begin{Prop}\label{nodal prop}
On a compact boundaryless spin surface $(M,\ig)$ of
genus $\ga$, suppose
that $H: M\to(0,\infty)$ is a smooth function and
$H_{max}:=\max_{\xi\in M}H(\xi)$.
Let $\psi$ be a solution of the equation
\[
D_\ig\psi=H|\psi|^2\psi \quad \text{on } M.
\]
Then the number of zeros of $\psi$ is at most
\[
\ga-1+\frac{\int_MH^2|\psi|^4d\vol_{\ig}}{4\pi}.
\]
In particular, if $M$ has $\ga=0$ (i.e. $M$ is
homeomorphic to the $2$-sphere) and
$\psi$ satisfies
\begin{\equ}\label{energy control}
\int_MH|\psi|^4d\vol_{\ig}<\frac{8\pi}{H_{max}},
\end{\equ}
then $\psi$ has no zero at all.
\end{Prop}
\begin{proof}
The proof is similar to that carried out in \cite[Proposition 8.1]{Ammann2009}. Here, for the sake of completeness, we will sketch the proof as follows.

On $M\setminus\psi^{-1}(0)$, let us introduce a new
metric $\ig_1=|\psi|^4\ig$. Then the transformation
formula for the Dirac operator under conformal changes
(Proposition \ref{conformal formula}) gives us that
there is a spinor field $\va_1$ on $M\setminus\psi^{-1}(0)$
such that
\[
D_{\ig_1}\va_1=H\va_1  \quad \text{and} \quad
|\va_1|_{\ig_1}=1.
\]
By Lemma \ref{Scal lemma}, we have the scalar curvature
can be estimated by $Scal_{\ig_1}\leq2H^2$. And hence,
we have the estimate for the Gau\ss \ curvature, that is
$K_{\ig_1}\leq H^2$. Moreover, we have
\begin{\equ}\label{vol}
\vol(M\setminus\psi^{-1}(0),\ig_1)=
\int_{M\setminus\psi^{-1}(0)}d\vol_{\ig_1}=\int_{M}
|\psi|^4d\vol_{\ig}.
\end{\equ}

Let $\psi^{-1}(0)=\{\xi_1,\dots,\xi_l\}$ for some $l\geq1$
if $\psi^{-1}(0)\neq\emptyset$. Let $n_j$ be the order
of the first non-vanishing term in the Taylor
expansion of $\psi$ near $\xi_j$, $j=1,\dots,l$.
Then it can be calculated that the integral of the geodesic
curvature $\ka_{\ig_1}$ over a small circle around
each $\xi_j$ is close to $-2\pi(2n_j+1)$. Now, we may
remove small open disks around $\xi_j$ from $M$,
and we obtain a surface $\tilde M$ with boundary.
By using the Gauss-Bonnet formula, we can derive
\[
2\pi\chi(\tilde M)=\int_{\tilde M}K_{\ig_1}d\vol_{\ig_1}
+\int_{\pa\tilde M}\ka_{\ig_1}ds
\leq\int_MH^2|\psi|^4d\vol_{\ig}-\sum_{j=1}^l
2\pi(2n_j+1).
\]
And hence, return to $M$ itself, we have
\[
2\pi(2-2\ga)=2\pi\chi(M)\leq
\int_MH^2|\psi|^4d\vol_{\ig}-4\pi\sum_{j=1}^l n_j,
\]
which implies $l$ is at most
$
\ga-1+\frac{\int_MH^2|\psi|^4d\vol_{\ig}}{4\pi}
$.

For the case $\ga=0$ and $\psi$
satisfies \eqref{energy control}, we assume
$\psi^{-1}(0)=\{\xi_1,\dots,\xi_l\}\neq\emptyset$.
Then we have
\[
1\leq l\leq -1+\frac{\int_MH^2|\psi|^4d\vol_{\ig}}{4\pi}
\leq-1+\frac{H_{max}}{4\pi}\int_MH|\psi|^4d\vol_{\ig}
<1
\]
which is impossible. Therefore $\psi$ has no zero at all.
\end{proof}

\section{Proof of Theorem \ref{main theorem1} and Corollary \ref{main theorem2}}\label{existence}
\subsection{The existence results}


For notation convenience, we will write $\cl$ instead of $\cl_{2^*}$. The super-level sets for $H$
will be simply denoted by
\[
\{H\geq c\}=\big\{\xi\in S^m:\, H(\xi)\geq c\big\}
\]
for any $c\in\R$. Recalling the hypothesis $(H*)$, it follows that
$H$ has only finitely many critical points at level $d$. 
And we may denote the corresponding critical set as
\[
\ck_{d}=\big\{\xi\in S^m:\, H(\xi)=d,\,
\nabla H(\xi)=0 \big\}=\{\xi_1,\dots,\xi_l\}
\]
for some $l\in\N$. We also introduce a
$\nu$-neighborhood of $\ck_{d}$ as
$\co_\nu:=\cup_{k=1}^l B_{\nu}(\xi_k)$  for $\nu>0$,
where $B_r(\xi)$ is the open ball centered at $\xi\in S^m$
with radius $r>0$ with respect to the metric $\ig_{S^m}$.
Standard deformation arguments show that hypothesis $(H*)$ ensures that $d$ is indeed a critical value for $H$, i.e. $\ck_d\neq\emptyset$, and $\ch$ is not contractible
in $\{H\geq d+\sa\}$ for any $\sa>0$. Moreover, for small $\nu>0$,
there exists $\sa>0$ such that $\ch$ is contractible in
$\{H\geq d+\sa\}\cup\co_\nu$.

\medskip

Now we construct the test
spinors that will be helpful to characterize the level sets
of the energy functional $I_{p_n}$ on the constraint manifold $\msn_{p_n}$ (see its definition in Subsection \ref{variational settings}). Let's start with an arbitrary $\phi_0\in\mbs_m$ (a constant spinor)
such that $|\phi_0|=\frac1{\sqrt2}\big(\frac m2
\big)^{\frac{m-1}2}$. Then, we define
\[
\phi(x)=f(x)^{\frac m2}(1-x)\cdot_{\ig_{\R^m}}\phi_0
\]
where $f(x)=\frac2{1+|x|^2}$ for $x\in\R^m$. It is
easy to verify that
\[
D_{\ig_{\R^m}}\phi=\frac m2f\phi
\]
and
\begin{\equ}\label{phi-norm}
|\phi|=\big(\frac m2\big)^{\frac{m-1}2}f^{\frac{m-1}2}.
\end{\equ}
For $\vr>0$, we set
\[
\phi_\vr(x)=\vr^{-\frac{m-1}2}\phi(x/\vr) \quad \text{and}
\quad \phi_{y,\vr}(x)=H(y)^{-\frac{m-1}2}\phi_\vr(x)
\]
where $y\in S^m$ is a fixed point. Then we have
$\phi_{y,\vr}$ satisfies the equation
\begin{\equ}\label{phi-y-vr-equ}
D_{\ig_{\R^m}}\phi_{y,\vr}=H(y)|\phi_{y,\vr}|^{2^*-2}
\phi_{y,\vr} \quad \text{on } \R^m.
\end{\equ}

To transplant $\phi_{y,\vr}$ on the sphere $S^m$, we need the following notation of stereographic projections. Given $y\in S^m$, we can embed $S^m$ into $\R^{m+1}$ in the way that its antipodal point $-y$ is the
North pole. Denoting $\cs_y: S^m\setminus\{-y\}\to\R^m$ the stereographic projection, we have $\cs_y(y)=0$. Moreover, $S^m\setminus\{-y\}$ and $\R^m$ are conformally equivalent due to the fact $(\cs_y^{-1})^*\ig_{S^m}=f^2\ig_{\R^m}$ with $f(x)=\frac2{1+|x|^2}$.

Recall the conformal transformation property mentioned in
Proposition \ref{conformal formula}, there is an
isomorphism of spinor bundles
$\iota:\mbs\big(\R^m, (\cs_\xi^{-1})^*\ig_{S^m}\big)\to
\mbs(\R^m,\ig_{\R^m})$ such that
\[
D_{\ig_{\R^m}}\big( \iota(\va) \big) = \iota\big( f^{\frac{m+1}2}
D_{(\cs_\xi^{-1})^*\ig_{S^m}} (f^{-\frac{m-1}2}\va)\big),
\]
where $D_{(\cs_\xi^{-1})^*\ig_{S^m}}$ is the Dirac
operator on $\R^m$ with respect to the metric
$(\cs_\xi^{-1})^*\ig_{S^m}$. Thus $\phi_{y,\vr}$ corresponds
to a spinor field $\psi_{y,\vr}$ on $S^m$ via the formula
\begin{\equ}\label{phi-y-vr--psi-y-vr}
\phi_{y,\vr}=\iota\big( f^{\frac{m-1}2}\psi_{y,\vr}
\circ\cs_y^{-1} \big)
\end{\equ}
such that $\psi_{y,\vr}$ satisfies the equation
\[
D\psi_{y,\vr}=H(y)|\psi_{y,\vr}|^{2^*-2}\psi_{y,\vr} \quad \text{on } (S^m,\ig_{S^m})
\]

\begin{Lem}\label{energy1}
For any $\sa>0$ and $\theta>0$
there exists $\vr>0$ such that
\[
\max_{t>0}I_{p_n}(t\psi_{y,\vr}^+)\leq
\frac1{2m\, (d+\sa)^{m-1}}\big(\frac m2\big)^{m}\om_m
+2\theta
\]
uniformly for all large $n$ and $y\in\{H\geq d+\sa\}$.
\end{Lem}
\begin{proof}
The strategy is to use Corollary \ref{continuity prop}.
And to begin with, we first fix $\al>0$ associated to $\theta$
as in Corollary \ref{continuity prop}.

In $E^*$, we have
\[
\cl'(\psi_{y,\vr})=
D\psi_{y,\vr}-H|\psi_{y,\vr}|^{2^*-2}\psi_{y,\vr}
=\big( H(y)-H \big)|\psi_{y,\vr}|^{2^*-2}\psi_{y,\vr}.
\]
Let $\psi\in E$ be such that $\|\psi\|\leq 1$, we have
\begin{eqnarray}\label{dd0}
\cl'(\psi_{y,\vr})[\psi]&=&
\real\int_{S^m}\big( H(y)-H \big)
|\psi_{y,\vr}|^{2^*-2}(\psi_{y,\vr},\psi)d\vol_{\ig_{S^m}}
\nonumber \\
&=&\real\int_{\R^m}\big( H(y)-H\circ\cs_y^{-1} \big)
|\phi_{y,\vr}|^{2^*-2}\big(\phi_{y,\vr}, \iota\big(
f^{\frac{m-1}2}\psi\circ\cs_y^{-1}\big)\big)
d\vol_{\ig_{\R^m}} \nonumber \\
&\leq& C \Big(\int_{\R^m}\big| H(y)-H\circ\cs_y^{-1}
\big|^{\frac{2^*}{2^*-1}}
|\phi_{y,\vr}|^{2^*}d\vol_{\ig_{\R^m}}
\Big)^{\frac{2^*-1}{2^*}}.
\end{eqnarray}
Fix $\de>0$ arbitrarily small, we may find that
$H\circ\cs_y^{-1}(x)=H(y)+O(\de)$ uniformly
for $|x|\leq\de$.
And hence, by the fact $\int_{\R^m}|\phi_\vr|^{2^*}
d\vol_{\ig_{\R^m}}\equiv\big(\frac m2\big)^m\om_m$,
we can get
\begin{\equ}\label{dd2}
\int_{B_\de^0}\big| H(y)-H\circ\cs_y^{-1}
\big|^{\frac{2^*}{2^*-1}}
|\phi_{y,\vr}|^{2^*}d\vol_{\ig_{\R^m}} \leq
O(\de^{\frac{2^*}{2^*-1}}).
\end{\equ}
On the other hand, we find that
\begin{\equ}\label{dd3}
\int_{\R^m\setminus B_\de^0}
\big| H(y)-H\circ\cs_y^{-1} \big|^{\frac{2^*}{2^*-1}}
|\phi_{y,\vr}|^{2^*}d\vol_{\ig_{\R^m}} \leq
C \int_{\frac\de\vr}^\infty \frac{r^{m-1}}{(1+r^2)^m}dr
\leq C \big(\frac\vr\de\big)^m.
\end{\equ}
Therefore, combining \eqref{dd0}, \eqref{dd2} and
\eqref{dd3}, we have
\begin{\equ}\label{dd4}
\|\cl'(\psi_{y,\vr})\|_{E^*}\leq \frac\al2 \quad
\text{uniformly for all } y\in S^m
\end{\equ}
provided that $\vr>0$ is fixed small enough.

Next, we estimate the upper bound of 
$\cl(\psi_{y,\vr})$ in a similar way, that is
\begin{eqnarray}\label{dd5}
\cl(\psi_{y,\vr})&=&\frac{H(y)}{2m}\int_{S^m}
|\psi_{y,\vr}|^{2^*}d\vol_{\ig_{S^m}}+\frac1{2^*}
\int_{S^m}\big( H(y)-H \big) |\psi_{y,\vr}|^{2^*}
d\vol_{\ig_{S^m}}  \nonumber\\
&=&\frac{H(y)}{2m}\int_{\R^m}
|\phi_{y,\vr}|^{2^*}d\vol_{\ig_{\R^m}}+\frac1{2^*}
\int_{\R^m}\big( H(y)-H\circ\cs_y^{-1} \big)
|\phi_{y,\vr}|^{2^*}d\vol_{\ig_{\R^m}}  \nonumber\\
&\leq&\frac{H(y)}{2m}\int_{\R^m}
|\phi_{y,\vr}|^{2^*}d\vol_{\ig_{\R^m}}+\frac\theta2
\end{eqnarray}
if $\vr$ is small enough. Moreover, by noticing that
\begin{\equ}\label{dd*}
\int_{\R^m}|\phi_{y,\vr}|^{2^*}d\vol_{\ig_{\R^m}}
=H(y)^{-m}\int_{\R^m}|\phi_\vr|^{2^*}d\vol_{\ig_{\R^m}}
=H(y)^{-m}\big(\frac m2\big)^m\om_m,
\end{\equ}
we soon obtain from \eqref{dd5} and $y\in\{H\geq d+\sa\}$ that
\begin{\equ}\label{dd6}
\cl(\psi_{y,\vr})\leq\frac1{2m H(y)^{m-1}}
\big(\frac m2\big)^m\om_m +\frac\theta2
\leq\frac1{2m\, (d+\sa)^{m-1}}
\big(\frac m2\big)^m\om_m +\frac\theta2
\end{\equ}
uniformly in $y$.

Remark that
the function $p\mapsto |\psi|_p$ is continuous provided that
$\psi\in E$ is fixed. Hence if we fix $\vr$ sufficiently small,
then it can be derived from \eqref{dd6} and the fact $\{H\geq d+\sa\}$ is compact that
\begin{\equ}\label{energy up}
0<c_1\leq\cl_{p_n}(\psi_{y,\vr})\leq \frac1{2m\, (d+\sa)^{m-1}}
\big(\frac m2\big)^m\om_m +\theta
\end{\equ}
uniformly for all large $n$, where $c_1>0$
is some constant.

It remains to evaluate the derivatives of $\cl_{p_n}$
for large $n$. Take $\psi\in E$ arbitrarily as a test spinor,
then it follows  that
\[
\aligned
&\cl_{p_n}'(\psi_{y,\vr})[\psi]-\cl'(\psi_{y,\vr})[\psi] \\
&\quad=\real\int_{S^m}H\big(|\psi_{y,\vr}|^{2^*-2}
-|\psi_{y,\vr}|^{p_n-2}\big)
(\psi_{y,\vr},\psi)d\vol_{\ig_{S^m}} \\
&\quad=\real\int_{\R^m}(H\circ\cs_y^{-1})\big(
|\phi_{y,\vr}|^{2^*-p_n}-f^{\frac{m-1}2(2^*-p_n)}\big)
|\phi_{y,\vr}|^{p_n-2}(\phi_{y,\vr}, \iota(f^{\frac{m-1}2}\psi
\circ\cs_y^{-1}))d\vol_{\ig_{\R^m}}.
\endaligned
\]
Remark that, by \eqref{phi-norm}, we have
\[
|\phi_{y,\vr}(x)|=H(y)^{-\frac{m-1}2}\cdot\vr^{-\frac{m-1}2}
\cdot\big(\frac m2\big)^{\frac{m-1}2}\cdot \Big(
\frac{2\vr^2}{\vr^2+|x|^2} \Big)^{\frac{m-1}2}.
\]
Thus
\[
|\phi_{y,\vr}|^{2^*-p_n}-f^{\frac{m-1}2(2^*-p_n)}
=C_\vr^{2^*-p_n}\Big(
\frac{2}{1+|x|^2} \Big)^{\frac{m-1}2(2^*-p_n)}\Big[
\Big( \frac{\vr^2+\vr^2|x|^2}{\vr^2+|x|^2}
\Big)^{\frac{m-1}2(2^*-p_n)}-1 \Big],
\]
where $C_\vr=H(y)^{-\frac{m-1}2}\cdot\vr^{-\frac{m-1}2}
\cdot\big(\frac m2\big)^{\frac{m-1}2}>0$ is nothing but a constant.
Since $\vr>0$ is fixed sufficiently small (say $\vr<1$),
we easily get
\[
\vr^2\leq\frac{\vr^2+\vr^2|x|^2}{\vr^2+|x|^2} \leq 1
\quad \text{for all } x\in\R^m.
\]
Therefore, due to $p_n\to2^*$, we obtain $C_\vr^{2^*-p_n}=1+o_n(1)$ and
\[
|\phi_{y,\vr}|^{2^*-p_n}-f^{\frac{m-1}2(2^*-p_n)}
=o_n(1)\Big(
\frac{2}{1+|x|^2} \Big)^{\frac{m-1}2(2^*-p_n)}
=o_n(1)f^{\frac{m-1}2(2^*-p_n)}
\]
uniformly in $y$ as $n\to\infty$. This implies
\[
\aligned
&\cl_{p_n}'(\psi_{y,\vr})[\psi]-\cl'(\psi_{y,\vr})[\psi]\\
&\quad=o_n(1)\real\int_{\R^m}(H\circ\cs_y^{-1})
f^{\frac{m-1}2(2^*-p_n)}
|\phi_{y,\vr}|^{p_n-2}(\phi_{y,\vr}, \iota(f^{\frac{m-1}2}\psi
\circ\cs_y^{-1}))d\vol_{\ig_{\R^m}}  \\
&\quad=o_n(1)\real\int_{S^m}H|\psi_{y,\vr}|^{p_n-2}
(\psi_{y,\vr},\psi)d\vol_{\ig_{S^m}}  \\
&\quad\leq o_n(1)\|\psi\|
\endaligned
\]
uniformly for all $y\in \{H\geq d+\sa\}$ as $n\to\infty$. Since
$\psi\in E$ is arbitrary, we can derive from \eqref{dd4} that
\begin{\equ}\label{derivetive-norm}
\|\cl_{p_n}'(\psi_{y,\vr})\|_{E^*}\leq
\|\cl'(\psi_{y,\vr})\|_{E^*}+\|\cl_{p_n}'(\psi_{y,\vr})
-\cl'(\psi_{y,\vr})\|_{E^*}\leq \al
\end{\equ}
uniformly for all $y\in \{H\geq d+\sa\}$  and large $n$.

Combining \eqref{energy up} and \eqref{derivetive-norm},
we may then apply Corollary \ref{continuity prop} to get
the desired assertion.
\end{proof}

The next result concerns with the upper bound estimate for $\displaystyle \max_{t>0}I_{p_n}(t\psi_{y,\vr}^+)$ where $y$ locates near the critical set $\ck_d$. Recall that, for $\nu>0$, we denote $\co_\nu$ the $\nu$-neighborhood of $\ck_d$.
\begin{Lem}\label{energy2}
There exist $C_0,\,\nu_0>0$ such that
\[
\max_{t>0}I_{p_n}(t\psi_{y,\vr}^+)\leq\frac1{2m\, d^{m-1}}\big(\frac m2
\big)^m\om_m
-\left\{
\aligned
&C_0 \vr^2|\ln\vr| + O(\vr^2) &\, & m=2,\\[0.4em]
&C_0 \vr^2 +O(\vr^\frac{m+3}2) & & m\geq 3,
\endaligned \right.
\]
uniformly for all small $\vr$, large $n$ and $y\in\co_{\nu_0}$.
Particularly, for all $y\in\co_{\nu_0}$, the estimate
\[
\max_{t>0}I_{p_n}(t\psi_{y,\vr_0}^+)\leq
\frac1{2m\, d^{m-1}}\big(\frac m2
\big)^m\om_m-\theta_0 \quad \text{for some } \theta_0>0
\]
holds as long as $\vr_0$ is small.
\end{Lem}
\begin{proof}
Since $\ck_{d}=\{\xi_1,\dots,\xi_l\}$
contains only finitely many points, without loss of generality,
from now on we will first restrict ourselves to a neighborhood around $\xi_1$.
Since the hypothesis $(H*)$ ensures the Hessian of $H$
at each $\xi\in\ck_{d}$ is positive definite.
Hence, there exits $R_1, \de_0>0$ such that
\begin{\equ}\label{key}
H(y)\geq d + \de_0\dist_{\ig_{S^m}}(y,\xi_1)^2
\end{\equ}
for all $y\in B_{2R_1}(\xi_1)\subset S^m$.

Analogous to \eqref{dd5},
let us estimate the following quantity
\[
\cl(\psi_{y,\vr})
=\frac{H(y)}{2m}\int_{\R^m}
|\phi_{y,\vr}|^{2^*}d\vol_{\ig_{\R^m}}
+\frac1{2^*}\int_{\R^m}\big( H(y)-H\circ
\cs_y^{-1} \big)
|\phi_{y,\vr}|^{2^*}d\vol_{\ig_{\R^m}}.
\]
And, by \eqref{dd*} and \eqref{key}, we find that
\begin{\equ}\label{key4-1}
\cl(\psi_{y,\vr})\leq\frac1{2m\, d^{m-1}}\big(\frac m2
\big)^m\om_m +\frac1{2^*}\int_{\R^m}\big( H(y)-H\circ
\cs_y^{-1} \big)
|\phi_{y,\vr}|^{2^*}d\vol_{\ig_{\R^m}}.
\end{\equ}
To estimate the second integral, let us recall that $H(y)=H\circ\cs_y^{-1}(0)$ and $y$ locates close to
$\xi_1$, hence we have the Hessian of $H\circ\cs_y^{-1}(\cdot)$ is positive definite in an open ball $B_{r_0}^0$, for some $r_0<R_1$ and $y\in B_{R_1}(\xi_1)$. In particular, we deduce that
\begin{\equ}\label{strongly convex}
2H(y)-H\circ\cs_y^{-1}(x)-H\circ\cs_y^{-1}(-x)\leq-2\de_0|x|^2
\quad \text{for all } x\in B_{r_0}^0.
\end{\equ}
At this point, let us split the last integral in \eqref{key4-1} into two parts
\[
\Big( \int_{B_{r_0}^0}+\int_{\R^m\setminus B_{r_0}^0}\Big)\big( H(y)-H\circ
\cs_y^{-1} \big)
|\phi_{y,\vr}|^{2^*}d\vol_{\ig_{\R^m}}.
\]
And it can be seen from a similar estimate in \eqref{dd3} that
\[
\int_{\R^m\setminus B_{r_0}^0}\big( H(y)-H\circ
\cs_y^{-1} \big)
|\phi_{y,\vr}|^{2^*}d\vol_{\ig_{\R^m}}=O(\vr^m).
\]
Notice that
$|\phi_{y,\vr}(x)|=|\phi_{y,\vr}(-x)|$ for all $x\in\R^m$, it follows from \eqref{strongly convex} that
\[
\aligned
&\int_{B_{r_0}^0}\big( H(y)-H\circ
\cs_y^{-1} \big)
|\phi_{y,\vr}|^{2^*}d\vol_{\ig_{\R^m}}  \\
&\qquad \leq-\de_0\int_{B_{r_0}^0}
|x|^2
|\phi_{y,\vr}|^{2^*}d\vol_{\ig_{\R^m}}  =-C_1\vr^2\int_0^{\frac{r_0}\vr}
\frac{r^{m+1}}{(1+r^2)^m}dr\leq\left\{
\aligned
& -C_1 \vr^2|\ln\vr| &\,& m=2,\\
&- C_1 \vr^2 & & m\geq3,
\endaligned \right..
\endaligned
\]
for all $\vr$ small,
where $C_1>0$ is a constant depending only on the
dimension $m$ and the value of $H(y)$. Therefore, by \eqref{key4-1}, we obtain immediately
\[
\cl(\psi_{y,\vr})\leq\frac1{2m\, d^{m-1}}\big(\frac m2
\big)^m\om_m+O(\vr^m)-\left\{
\aligned
& C_1 \vr^2|\ln\vr| &\,& m=2,\\
& C_1 \vr^2 & & m\geq3,
\endaligned \right.
\]
uniformly for $y\in B_{R_1}(\xi_1)$ and for $\vr$ small.

On the other hand, for the derivative of $\cl$, we
shall use \eqref{dd0} to get
\[
\|\cl'(\psi_{y,\vr})\|_{E^*}\leq C
\Big(\int_{\R^m}\big| H(y)-H\circ\cs_{y}^{-1}
\big|^{\frac{2^*}{2^*-1}}
|\phi_{y,\vr}|^{2^*}d\vol_{\ig_{\R^m}}
\Big)^{\frac{2^*-1}{2^*}}
\]
for all $\vr>0$. Then, it follows that
\begin{\equ}\label{cl-der}
\|\cl'(\psi_{y,\vr})\|_{E^*}^{\frac{2^*}{2^*-1}}
\leq C|\nabla H|_{L^\infty(B_{r_0}(y))}^{\frac{2m}{m+1}}
\int_{B_{r_0}^0}|x|^{\frac{2m}{m+1}}
|\phi_{y,\vr}|^{2^*}d\vol_{\ig_{\R^m}}+O(\vr^m)
\end{\equ}
as $\vr\to0$. Remark that
\begin{\equ}\label{x2}
\int_{\R^m}|x|^{\frac{2m}{m+1}}
|\phi_{y,\vr}|^{2^*}d\vol_{\ig_{\R^m}}
= O(\vr^{\frac{2m}{m+1}})
\end{\equ}
as $\vr\to0$. Combining \eqref{cl-der}
and \eqref{x2}, we conclude that
\[
\|\cl'(\psi_{y,\vr})\|_{E^*}\leq
C_2 |\nabla H|_{L^\infty(B_{r_0}(y))} \cdot \vr
+O(\vr^{\frac{m+1}2})
\]
as $\vr\to0$ with some constant $C_2>0$.

\medskip

Now, given $\vr>0$ small enough, 
we see that  there exists $n_\vr\in\N$
such that for all $n\geq n_\vr$
\begin{\equ}\label{x3}
\cl_{p_n}(\psi_{y,\vr})
\leq\frac1{2m\, d^{m-1}}\big(\frac m2
\big)^m\om_m +O(\vr^m) -\left\{
\aligned
& \frac{C_1}2 \vr^2|\ln\vr| &\,& m=2,\\
& \frac{C_1}2 \vr^2 & & m\geq3,
\endaligned \right.
\end{\equ}
and
\begin{\equ}\label{x4}
\|\cl_{p_n}'(\psi_{y,\vr})\|_{E^*}\leq
2C_2 |\nabla H|_{L^\infty(B_{r_0}(y))} \cdot \vr
+O(\vr^{\frac{m+1}2})
\end{\equ}
uniformly for $y\in B_{R_1}(\xi_1)$.
Since $\nabla H(\xi_1)=0$, we can derive that
\[
|\nabla H|_{L^\infty(B_{r_0}(y))}\to0 \quad
\text{as } y\to\xi_1 \text{ and } r_0\to0.
\]
And therefore, by combining \eqref{x3}, \eqref{x4} and
Lemma \ref{key2}, we soon have
\[
\max_{t>0}I_{p_n}(t\psi_{y,\vr}^+)\leq\frac1{2m\, d^{m-1}}\big(\frac m2
\big)^m\om_m
-\left\{
\aligned
&\frac{C_1}4 \vr^2|\ln\vr| + O(\vr^2) &\, & m=2,\\[0.4em]
&\frac{C_1}4 \vr^2 +O(\vr^\frac{m+3}2) & & m\geq 3,
\endaligned \right.
\]
uniformly for $y\in B_{R_1}(\xi_1)$ 
provided that $R_1$ and $r_0$ are fixed small enough.

If $\ck_{d}=\{\xi_1,\dots,\xi_l\}$ contains multiple points. By repeating the
above argument, one shall get correspondingly $R_1,\dots,R_l>0$. And
the proof will be completed by simply taking $\nu_0=
\min\{R_1,\dots,R_l\}$.
\end{proof}

Now we are ready to proof Theorem \ref{main theorem1}

\begin{proof}[Proof of Theorem \ref{main theorem1}]
The proof will be accomplished by the following four steps:

\textbf{Step 1.} Recall the fact that, by the hypothesis $(H*)$,
for any $\nu>0$ small there exists $\sa>0$ such that $\ch$ is
contractible in $\{H\geq d+\sa\}\cup\co_\nu$. Now let
us fix $\nu_0>0$ as in Lemma \ref{energy2}, then we get
some $\sa_0>0$ such that $\ch$ is contractible in
$\cc_0:=\{H\geq d+\sa_0\}\cup\co_{\nu_0}$.
At this point, we may choose $\vr_0,\,\theta_0,\, \theta>0$
small such that
\[
\max\Big\{
\frac1{2m(d+\sa_0)^{m-1}}\big(\frac m2\big)^m\om_m,\
\frac1{2m\, d^{m-1}}\big(\frac m2\big)^m\om_m
-\theta_0 \Big\} \leq
\frac1{2m\, d^{m-1}}\big(\frac m2\big)^m\om_m
-3\theta
\]
where $\theta_0$ depends on $\vr_0$
(by Lemma \ref{energy2}).

Recall the the Nehari-Pankov manifold $\msn_{p_n}$
defined for $I_{p_n}$ (see \eqref{nehari}), let us
introduce the sub-level sets for $I_{p_n}$ as
\[
I_{p_n}^c=\big\{ u\in\msn_{p_n}:\, I_{p_n}(u)\leq c\big\}
\]
for $c\in\R$. Denoted by
%
%
\begin{\equ}\label{a}
\hat a=\frac1{2m\, d^{m-1}}\big(\frac m2\big)^m\om_m
-\theta,
\end{\equ}
then a combination of Lemma \ref{energy1} and
Lemma \ref{energy2} (we fix $\vr>0$ small enough such that both the lemmas hold) gives us that there exists a continuous embedding $\Phi_n:\cc_0\to I_{p_n}^{\hat a}$ for all $n$ large. We also remark that $\hat a<2\tau_{2^*}-\theta$.

\medskip

\textbf{Step 2.} Due to the hypothesis $(H*)$, we have
$H_{min}<d$. Hence if we fix $\xi_0\in S^m$
such that $H(\xi_0)=H_{min}$, then $\xi_0\not\in
\{H\geq d\}$.
Let $\cs_0: S^m\setminus\{\xi_0\}\to\R^m$
be the stereographic projection from $\xi_0$ (that is, we
treat $\xi_0$ as the North pole). We see that the
image of $\cc_0$ under $\cs_0$ is a bounded set in $\R^m$,
and thus there exists $R_0>0$ such that
\[
\cs_0(\cc_0)\subset B_{R_0}^0=\big\{x\in\R^m:\,
|x|<R_0 \big\}.
\]
Now we can introduce
a barycenter-type function $\Upsilon: E\setminus\{0\}\to\R^m$ as
\[
\Upsilon(\psi)=
\frac{\int_{S^m}\zeta\circ\cs_0(\xi)|\psi|^{2^*}
d\vol_{\ig_{S^m}}}{\int_{S^m}|\psi|^{2^*}
d\vol_{\ig_{S^m}}},
\]
where $\zeta:\R^m\to\R^m$ is defined by
\[
\zeta(x)=\left\{
\aligned
&x  &\,& |x|<R_0, \\
&\frac{R_0 x}{|x|} & & |x|\geq R_0.
\endaligned \right.
\]

Next, we claim that $\theta>0$ can be chosen
sufficiently small such that for all large $n$
and $u\in\msn_{p_n}$ satisfying
\[
I_{p_n}(u)<\tau_{2^*}+\theta
=\frac1{2m (H_{max})^{m-1}}\big(\frac m2\big)^m\om_m
+\theta,
\]
there holds
\begin{\equ}
\cs_0^{-1}\circ\Upsilon\big(u+h_{p_n}(u)\big)\in \ch_\de.
\end{\equ}

Suppose to the contrary that there exist $\theta_n\to0$ and
$u_n\in\msn_{p_n}$ such that
\begin{\equ}\label{contradiction X}
I_{p_n}(u_n)<\tau_{2^*}+\theta_n
\quad \text{but} \quad
\cs_0^{-1}\circ\Upsilon(\psi_n)\not\in \ch_\de,
\end{\equ}
where we set $\psi_n:=u_n+h_{p_n}(u_n)$. Recall the definition
of $\tau_p$ for $p\in(2,2^*]$ (see \eqref{tau p def}), by
Lemma \ref{tau p properties} we can get
\[
I_{p_n}(u_n)\geq
\tau_{p_n}
=\tau_{2^*}+o_n(1).
\]
Thus, as $n\to\infty$, each $u_n$ almost minimize the functional $I_{p_n}$ on $\msn_{p_n}$. And it follows from the well-known Ekeland Variational Principle (\cite{Ekeland, MW, Willem}) that for each $n$ there exists
$\tilde u_n\in\msn_{p_n}$ such that
\begin{\equ}\label{XX}
I_{p_n}(\tilde u_n)<\tau_{2^*}+2\theta_n,
\quad
\big\|I_{p_n}'\big|_{\msn_{p_n}}(\tilde u_n)\big\|
\leq \frac{8\theta_n}\epsilon
\quad \text{and} \quad
\|u_n-\tilde u_n\|\leq 2\epsilon.
\end{\equ}
where $\epsilon>0$ is a given small constant whose value will be fixed later. Remark that we may simply rewrite the definition of $\msn_{p_n}$ in \eqref{nehari} as
$\msn_{p_n}=\big\{ u\in E^+:\, \ck_{p_n}(u)=0 \big\}$ with $\ck_{p_n}(u)=I_{p_n}'(u)[u]$. Then the Lagrange multiplier rule implies the existence of  $t_n\in\R$ such that
\[
I_{p_n}'\big|_{\msn_{p_n}}(\tilde u_n)
=I_{p_n}'(\tilde u_n) + t_n \ck_{p_n}'(\tilde u_n).
\]
By using the estimate
$\ck_{p_n}'(\tilde u_n)[\tilde u_n]\leq -\frac{p_n-2}{p_n-1}
\int_{S^m}H|\tilde u_n+h_{p_n}(\tilde u_n)|^{p_n} d\vol_{\ig_{S^m}}$ (a similar inequality has been shown in \cite[Lemma 7.4]{BX}),
we can easily conclude from \eqref{XX} that $t_n\to0$
and, hence, $I_{p_n}'(\tilde u_n)\to0$ as $n\to\infty$.
Now, set $\tilde\psi_n=\tilde u_n+h_{p_n}(\tilde u_n)$,
it is evident that
\[
\tau_{2^*}+o_n(1)\leq
\cl_{p_n}(\tilde\psi_n) <\tau_{2^*}
+2\theta_n \quad \text{and} \quad
\cl_{p_n}'(\tilde\psi_n)\to0
\]
as $n\to\infty$. Then, according to $(2)$ of
Remark \ref{ground state remark} and
Proposition \ref{alternative prop}, we have the sequence
$\{\tilde\psi_n\}$ must blow-up. Moreover, by 
Proposition \ref{blow-up prop}, there exists
a convergent sequence $\{a_n\}\subset S^m$,
$a_n\to a\in \ch$ as $n\to\infty$ and a
sequence of positive numbers $R_n\searrow0$ and a non-trivial solution $\phi_0$ of
\[
D_{\ig_{\R^m}}\phi_0=H_{max}|\phi_0|^{2^*-2}\phi_0
\quad \text{on } \R^m
\]
such that 
\[
\tilde\psi_n=R_n^{-\frac{m-1}2}\eta(\cdot)\ov{(\mu_n)}_*
\circ\phi_0\circ\mu_n^{-1}+o_n(1) \quad \text{in } E
\]
as $n\to\infty$, where $\mu_n$
and $\eta$ are as in Proposition \ref{blow-up prop}.
Therefore, $\int_{S^m}|\tilde\psi_n|^{2^*}
d\vol_{\ig_{S^m}}$ is bounded away from $0$ and particularly
\[
\Upsilon(\tilde\psi_n)\to \cs_0(a)\in \cs_0(\ch)
\quad \text{as } n\to\infty.
\]

Recalling the definition of $\Upsilon$, we see that the derivative
$\Upsilon': E\setminus\{0\}\to E^*$ exists and the operator norm is bounded around each $\tilde\psi_n$. Hence, by fixing the constant $\epsilon$ 
small enough in \eqref{XX}, we have $|\Upsilon(\psi_n)-\Upsilon(\tilde\psi_n)|<\frac\de2$, and therefore we get
$\Upsilon(\psi_n)\in \cs_0(\ch_\de)$
for all $n$ large which contradicts to \eqref{contradiction X}.

\medskip

\textbf{Step 3.} We are going to show that $I_{p_n}$
has a critical point in $I_{p_n}^{\hat a}\setminus I_{p_n}^{\hat b}$
where $\hat a$ is defined in \eqref{a} and
\[
\hat b=\frac1{2m (H_{max})^{m-1}}\big(\frac m2\big)^m\om_m
+\theta.
\]
We emphasize here that the embedding $\Phi_n:\cc_0\to\msn_{p_n}$ gives us a description of low-energy levels in the sense that $\Phi_n(\ch)\in I_{p_n}^{\hat b}$ and $\cs_0^{-1}\circ\Upsilon\circ\Phi_n|_\ch$ is homotopic to
the inclusion $\ch\hookrightarrow \ch_\de$.
Fixing all these notations, we define the family 
\[
\Lam_n=\Big\{\varpi\in C\big(\ch\times[0,1],\msn_{p_n}\big):\, \varpi(\cdot,0)=\Phi_n \text{ and } \varpi(\cdot,1) \text{ collapses to a single spinor}   \Big\}.
\]

Clearly, for all $n$ large, $\Lam_n\neq\emptyset$ since the contractibility of $\ch$ in $\cc_0$ gives us a continuous homotopy
$\ga:\ch\times[0,1]\to \cc_0$ such that
$\ga(\cdot,0)=id$, $\ga(\cdot,1)\equiv\xi_*\in\cc_0$ and $\Phi_n\circ\ga\in\Lam_n$. Since $E$ is
embedded compactly into $L^{p}$ for all $p\in[1,2^*)$,
we have $I_{p_n}$ satisfies the Palais-Smale condition
for each $n$.
And thus, by standard variational arguments, one can easily verify
\[
c_n:=\inf_{\varpi\in\Lam_n}\sup_{(\xi,t)\in\ch\times[0,1]}I_{p_n}(\varpi(\xi,t))
\]
is a critical value for $I_{p_n}$. Particularly, $c_n\in(\hat b,\hat a]$.

\medskip

\textbf{Step 4.} Let $\psi_n^+$ be a critical point of $I_{p_n}$ at level $c_n$. It suffices to show that the sequence $\{\psi_n=\psi_n^++h_{p_n}(\psi_n^+)\}$ contains a compact subsequence.

Let us assume to the contrary that any of such sequence $\{\psi_n\}$ will
blow up and let $a\in S^m$ be the associated blow-up point. According to Proposition \ref{blow-up prop}
and the fact $\cl_{p_n}(\psi_n)=I_{p_n}(\psi_n^+)=c_n\leq \hat a$,
the blow-up point $a$ should locate inside $\{H\geq d+\sa_0\}$. By hypothesis $(H*)$, there is no critical value of $H$ in the interval $(d,H_{max})$ and hence we have $a\in \ch$.

Moreover, from Proposition \ref{blow-up prop} again, there is a non-trivial solution $\phi_0$ of
\[
D_{\ig_{\R^m}}\phi=H_{max}|\phi|^{2^*-2}\phi
\quad \text{on } \R^m
\]
such that
\[
\psi_n=R_n^{-\frac{m-1}2}\eta(\cdot)\ov{(\mu_n)}_*
\circ\phi_0\circ \mu_n^{-1}+o_n(1) \quad \text{in } E
\]
as $n\to\infty$, where $R_n$, $\eta$, $\mu_n$ are as in Proposition \ref{blow-up prop}. Moreover
\[
\int_{S^m}H|\psi_n|^{p_n}d\vol_{\ig_{S^m}}=R_n^{\frac{m-1}2(2^*-p_n)}\int_{B_{2r/R_n}^0}H\circ\mu_n|\phi_0|^{p_n}d\vol_{\ig_n}+o_n(1)
\]
as $n\to\infty$. Since for arbitrary $R>0$
\[
\aligned
&R_n^{\frac{m-1}2(2^*-p_n)}\int_{B_{2r/R_n}^0\setminus B_R^0}H\circ\mu_n|\phi_0|^{p_n}d\vol_{\ig_n}\\
&\quad \leq
C\Big( \int_{B_{2r/R_n}^0\setminus B_R^0}|\phi_0|^{2^*} d\vol_{\ig_{\R^m}} \Big)^{\frac{p_n}{2^*}}\big( (2r)^m-(R_nR)^m \big)^{\frac{2^*-p_n}{2^*}}
\endaligned
\]
and
\[
R_n^{\frac{m-1}2(2^*-p_n)}\int_{B_R^0}H\circ\mu_n|\phi_0|^{p_n}d\vol_{\ig_n}
=H_{max}\int_{B_R^0}|\phi_0|^{2^*}d\vol_{\ig_n}+o_n(1),
\]
we derive that
\begin{\equ}\label{energy id}
\int_{S^m}H|\psi_n|^{p_n}d\vol_{\ig_{S^m}}=H_{max}\int_{\R^m}|\phi_0|^{2^*}d\vol_{\ig_n}+o_n(1)
\end{\equ}
as $n\to\infty$. Notice that $\phi_0$ also extends to a non-trivial solution $\bar\phi_0$ to the equation
\begin{\equ}\label{XX2}
D\bar\phi_0
=H_{max}|\bar\phi_0|^{2^*-2}\bar\phi_0
\quad \text{on } S^m,
\end{\equ}
by using the fact $\cl_{p_n}(\psi_n)=c_n\leq \hat a<2\tau_{2^*}$, we have
\[
H_{max}\int_{S^m}|\bar\phi_0|^{2^*}d\vol_{\ig_{S^m}}=
H_{max}\int_{\R^m}|\phi_0|^{2^*}d\vol_{\ig_{\R^m}}
<\frac2{(H_{max})^{m-1}}\big(\frac{m}2\big)^m\om_m.
\]

In the 2-dimensional case, we have $2^*=4$ and
\[
H_{max}\int_{S^2}|\bar\phi_0|^{4}d\vol_{\ig_{S^2}}=
H_{max}\int_{\R^2}|\phi_0|^{4}d\vol_{\ig_{\R^2}}
<\frac{8\pi}{H_{max}}.
\]
It follows from Proposition \ref{nodal prop} that $|\bar\phi_0|$ has no zero at all, and from the spinorial Weierstra\ss\ representation and Li-Yau's inequality that $(S^2, |\bar\phi_0|^4\ig_{S^2})$ is embedded into $\R^3$ with constant mean curvature $H_{max}$. Now, by the Alexandrov's theorem \cite{Alexandrov2}, we know that $(S^2, |\bar\phi_0|^4\ig_{S^2})$ must be a round sphere. And hence the Willmore energy satisfies
\[
H_{max}^2\int_{S^2}|\bar\phi_0|^{4}d\vol_{\ig_{S^2}}=4\pi.
\]
This and \eqref{energy id} imply
\[
c_n=\cl_{p_n}(\psi_n)=\frac{p_n-2}{2p_n}\int_{S^2}H|\psi_n|^{p_n}d\vol_{\ig_{S^2}}
=\frac{H_{max}}4\int_{\R^2}|\phi_0|^4d\vol_{\ig_{\R^2}}+o_n(1)<\hat b
\]
as $n\to\infty$, which is impossible. And therefore, the compactness of the solution sequence $\{\psi_n\}$ follows.
\end{proof}

\begin{Rem}
The hypersurfaces with constant mean curvature are much more complicated in higher dimensions, surprising examples have been shown in \cite{HTY}: there exist infinitely many distinct differentiable immersions of the $3$-sphere into Euclidean $4$-space having a given positive constant mean curvature, moreover, the total ``area" as well as the total integral of the norm of the second fundamental form of such examples can be as large as one wants. This makes the picture of $\bar\phi_0$ in Step 4 unclear to us when the dimension $m\geq3$, particularly we do not know whether or not $\bar\phi_0$ has the ground state energy. And it is also unclear if $\tau_{2^*}$ is an isolated critical level of the energy functional for Eq. \eqref{XX2}. Our approach in this regard up to now have failed.
\end{Rem}

\subsection{Application to 2-sphere in Euclidean 3-space}

Thanks to the previous section, we have the existence
result for the equation
\begin{\equ}\label{eeee}
D\psi=H|\psi|^{2}\psi \quad \text{on } S^2
\end{\equ}
provided $H$ is smooth and satisfy $(H)$.
Now let us give a interesting application to the problem
of prescribing mean curvature on the 2-sphere.

As a direct application of Proposition \ref{nodal prop}, we soon have

\begin{Cor}\label{existence positive solu}
On $(S^2,\ig_{S^2})$, if $H$ is a positive smooth
function satisfying criteria $(H)$, then there exists a solution of the equation
\[
D\psi=H|\psi|^2\psi, \quad
|\psi|>0.
\]
\end{Cor}

Following the spinorial Weierstra\ss\ representation\cite[Theorem 13]{Friedrich1998}, we apply our existence reults now to the problem of prescribing mean curvature on $S^2$:

\begin{Cor}\label{S2 immersed}
On $(S^2,\ig_{S^2})$,
if $H$ is a positive smooth function satisfying
criteria $(H)$, then there is a conformal metric $\ig_1$
such that $(S^2,\ig_1)$  is isometrically immersed into
the Euclidean space $\R^3$ with mean curvature $H$.
\end{Cor}
\begin{proof}
By Corollary \ref{existence positive solu}, we can introduce
a conformal metric $\ig_1=|\psi|^4\ig_{S^2}$ on $S^2$.
Then the conformal transformation formula implies that
there exists a spinor $\va_1$ on $(S^2,\ig_1)$ such that
\[
D_{\ig_1}\va_1=H\va_1 \quad \text{and} \quad
|\va_1|_{\ig_1}\equiv1.
\]
Hence by the spinorial Weierstra\ss\ representation, we have there is an isometric immersion $(\widetilde S^2,\ig)\to\R^3$ of the universal covering $\widetilde S^2$ into the Euclidean space $\R^3$ with mean
curvature $H$. Since $S^2$ is simply connected, we conclude the assertion immediately.
\end{proof}

\section{Appendix}
\subsection{Proof of Proposition \ref{alternative prop}}
The boundedness of $\{\psi_n\}$ follows from Lemma \ref{key2} (or one may follow \cite[Lemma 4.1]{BX}). We then assume without loss of generality that $\psi_n\rightharpoonup \psi_0$ in $E$ as $n\to\infty$. By \eqref{key-assumption}, it is not difficult to check that 
\[
\aligned
0&=\lim_{n\to\infty}\cl_{p_n}'(\psi_n)[\va]=\lim_{n\to\infty}\Big(\inp{\psi_n^+}{\va^+}-\inp{\psi_n^-}{\va^-}-\real\int_{S^m}H|\psi_n|^{p_n-2}(\psi_n,\va)d\vol_{\ig_{S^m}}\Big) \\
&=\inp{\psi_0^+}{\va^+}-\inp{\psi_0^-}{\va^-}-\real\int_{S^m}H|\psi_0|^{2^*-2}(\psi_0,\va)d\vol_{\ig_{S^m}}=\cl_{2^*}'(\psi_0)[\va]
\endaligned
\]
for all $\va\in E$. Hence $\psi_0$ turns out to be a solution of the critical equation \eqref{Dirac m}.

Set $\bar\psi_n=\psi_n-\psi_0$. We derive that
\[
\aligned
D\bar\psi_n&=H|\psi_n|^{p_n-2}\psi_n
-H|\bar\psi_n|^{p_n-2}\bar\psi_n
-H|\psi_0|^{p_n-2}\psi_0 \\[0.4em]
&\qquad +H|\psi_0|^{p_n-2}\psi_0
-H|\psi_0|^{2^*-2}\psi_0\\[0.4em]
&\qquad +H|\bar\psi_n|^{p_n-2}\bar\psi_n+o_n(1)
\endaligned
\]
where the above $o_n(1)$ term equals to $\cl_{p_n}'(\psi_n)$ as $n\to\infty$ in $E^*$. Denoted by
\[
\Phi_n=H|\psi_n|^{p_n-2}\psi_n
-H|\bar\psi_n|^{p_n-2}\bar\psi_n
-H|\psi_0|^{p_n-2}\psi_0,
\]
since $p_n\nearrow 2^*$, it turns out that there exists $C>0$ (independent of $n$) such that
\begin{\equ}\label{e4}
	|\Phi_n|\leq C|\bar\psi_n|^{p_n-2}|\psi_0|
	+C|\psi_0|^{p_n-2}|\bar\psi_n|.
\end{\equ}
Thanks to the Egorov theorem, for any $\epsilon>0$, there exists $\Om_\epsilon\subset S^m$ such that
$\meas\{S^m\setminus \Om_\epsilon\}<\epsilon$ and $\bar\psi_n\to0$ uniformly on $\Om_\epsilon$ as $n\to\infty$. Therefore, by the boundedness of $\psi_n$, \eqref{e4} and the H\"older inequality, we have
\begin{eqnarray}\label{integral1}
	\real\int_{S^m}(\Phi_n,\va)d\vol_{\ig_{S^m}}&=&
	\real\int_{S^m\setminus \Om_\epsilon}
	(\Phi_n,\va)d\vol_{\ig_{S^m}}
	+\real\int_{\Om_\epsilon}(\Phi_n,\va)d\vol_{\ig_{S^m}}
	\nonumber\\
	&\leq&C \Big[
	\Big(\int_{S^m\setminus \Om_\epsilon}
	|\psi_0|^{2^*}d\vol_{\ig_{S^m}}
	\Big)^{\frac1{2^*}}  
+\Big(\int_{S^m\setminus \Om_\epsilon}
	|\psi_0|^{2^*}d\vol_{\ig_{S^m}}
	\Big)^{\frac{p_n-2}{2^*}}\Big]
	  \nonumber\\
	& &+\int_{\Om_\epsilon}|\Phi_n|\cdot|\va|
	d\vol_{\ig_{S^m}}.
\end{eqnarray}
for arbitrary $\va\in E$ with $\|\va\|\leq1$. It is
evident that
the last integral in \eqref{integral1} converges to $0$ as
$n\to\infty$ and the remaining integrals tends to $0$
uniformly in $n$ as $\epsilon\to0$. Thus, we get
$\Phi_n=o_n(1)$ in $E^*$. Noting that
$q\mapsto H(\cdot)|\psi_0|^{q-2}\psi_0$ is continuous
in $E^*$,  we obtain
\begin{\equ}\label{eq1}
	\cl_{p_n}'(\bar\psi_n)=D\bar\psi_n-H|\bar\psi_n|^{p_n-2}\bar\psi_n=o_n(1)
	\quad \text{in } E^*.
\end{\equ}

Now assume $\psi_0\neq0$ in $E$ (otherwise we are done). Up to a subseqence, if $\cl_{p_n}(\bar\psi_n)\to0$ then, by repeating the proof in \cite[Lemma 4.1]{BX}, we have that  $\bar\psi_n\to0$ in $E$. This shows the compactness of $\{\psi_n\}$. In what follows, we assume that $\cl_{p_n}(\bar\psi_n)$ is bounded away from $0$.

Since $\psi_0$ is a non-trivial solution of Eq. \eqref{Dirac m}, we can see that
\begin{\equ}\label{X3}
\tau_{2^*}\leq\cl_{2^*}(\psi_0)=\cl_{2^*}(\psi_0)-\frac1{2^*}\cl_{2^*}'(\psi_0)[\psi_0]=\frac1{2m}\big( \|\psi_0^+\|^2-\|\psi_0^-\|^2 \big).
\end{\equ}
On the other hand, by \eqref{eq1} and  $\cl_{p_n}(\bar\psi_n)$ is bounded away from $0$, we derive that $\{\bar\psi_n\}$ satisfy the condition \eqref{contradiction-assumption} (where the upper bound of $\cl_{p_n}(\bar\psi_n)$ follows directly from the boundedness of $\{\bar\psi_n\}$). And hence, by applying Lemma \ref{key2}, we have
$\tau_{p_n}\leq\max_{t>0}I_{p_n}(t\bar\psi_n)\leq \cl_{p_n}(\bar\psi_n)+o_n(1)$.
This and \eqref{eq1} imply
\begin{\equ}\label{X4}
\tau_{p_n}\leq \cl_{p_n}(\bar\psi_n)-\frac1{p_n}\cl_{p_n}'(\bar\psi_n)[\bar\psi_n]+o_n(1)=\frac{p_n-2}{2p_n}\big( \|\bar\psi_n^+\|^2-\|\bar\psi_n^-\|^2 \big)+o_n(1).
\end{\equ}
Now, it follows from \eqref{X3}, \eqref{X4} and Lemma \ref{tau p properties} that
\[
\aligned
\cl_{p_n}(\psi_n)&=\cl_{p_n}(\psi_n)-\frac1{p_n}\cl_{p_n}'(\psi_n)[\psi_n]+o_n(1)=\frac{p_n-2}{2p_n}
\big( \|\psi_n^+\|^2-\|\psi_n^-\|^2\big) + o_n(1)\\
&=\frac{p_n-2}{2p_n}
\big( \|\bar\psi_n^+\|^2-\|\bar\psi_n^-\|^2\big)
+\frac{p_n-2}{2p_n}
\big(\|\psi_0^+\|^2-\|\psi_0^-\|^2\big)
+ o_n(1) \\
&\geq 2\tau_{2^*}+o_n(1)
\endaligned
\]
which contradicts \eqref{key-assumption}. Therefore, we must have  $\cl_{p_n}(\bar\psi_n)\to0$ as $n\to\infty$, and the proof is hereby completed.

\subsection{Proof of Proposition \ref{blow-up prop}}

We follow the steps in \cite[Section~5]{Isobe-JFA}, with  necessary modifications.

To begin with, for some $\de_0>0$, we introduce the singular set of $\{\psi_n\}$:
\[
\Ga:=\Big\{
a\in S^m:\, \varliminf_{r\to0}\varliminf_{n\to\infty}
\int_{B_r(a)}|\psi_n|^{p_n}d\vol_{\ig_{S^m}}\geq \de_0
\Big\}
\]
where $B_r(a)\subset S^m$ is the distance ball of radius $r$
with respect to the metric $\ig_{S^m}$. As was proved in \cite[Lemma 5.3]{Isobe-JFA} (the proof can be repeated  just adapting the notations in the present paper) that there exists $\de_0>0$ depending only on the geometry of $S^m$ such that $\Ga\neq\emptyset$.

In order to have a clearer picture of $\psi_n$ near points in $\Ga$, for $r\geq0$, we define
\[
\Theta_n(r)=\sup_{a\in S^m}\int_{B_r(a)}|\psi_n|^{p_n}
d\vol_{\ig_{S^m}},
\]
which can be understood as a concentration function for $\psi_n$. Then, by choosing $\bar\de>0$ small (say $\bar\de<\de_0$), there
exist a decreasing sequence $R_n\searrow0$
as $n\to\infty$ and $\{a_n\}\subset S^m$ such that
\begin{\equ}\label{choice of an}
	\Theta_n(R_n)=\int_{B_{R_n}(a_n)}|\psi_n|^{p_n}
	d\vol_{\ig_{S^m}}=\bar\de.
\end{\equ}
Up to a subsequence if necessary, we assume that
$a_n\to a\in S^m$ as $n\to\infty$.

Recall that we have defined the rescaled geodesic normal coordinates near each $a_n$ by $\mu_n(x)=\exp_{a_n}(R_nx)$ for $x\in\R^m$. Then we have a conformal equivalence $\R^m\supset(B_R^0, \, \ig_n)\cong (B_{R_nR}(a_n),\, \ig_{S^m})\subset S^m$ for all large $n$, where $\ig_n=R_n^{-2}\mu_n^*\ig_{S^m}$. And, for any $R>0$, $\ig_n$ converges to the Euclidean metric in $C^\infty(B_R^0)$ as $n\to\infty$. 

Following the idea of local trivialization introduced
in \cite{BG}, the coordinate map $\mu_n$ induces a spinor identification $\ov{(\mu_n)}_*:\mbs_x(B_R^0,\ig_n)\to
\mbs_{\mu_n(x)}(B_{R_nR}(a_n),\ig_{S^m})$. In the sequel, we define spinors $\phi_n$ on $B_R^0$ by
$\phi_n=R_n^{\frac{m-1}2}\ov{(\mu_n)}_*^{\,-1}\circ\psi_n\circ\mu_n$. By the conformal changes of the Dirac operator, a straightforward calculation shows that
$
	D_{\ig_n}\phi_n=R_n^{\frac{m+1}2}\ov{(\mu_n)}_*^{\,-1}
	\circ(D\psi_n)\circ\mu_n
$,
\begin{\equ}\label{c2}
	\int_{B_R^0}(D_{\ig_n}\phi_n,\phi_n)d\vol_{\ig_n}
	=\int_{B_{R_nR}(a_n)}(D\psi_n,\psi_n)d\vol_{\ig_{S^m}},
\end{\equ}
\begin{\equ}\label{c3}
	\int_{B_R^0}|\phi_n|^{2^*}d\vol_{\ig_n}=
	\int_{B_{R_nR}(a_n)}|\psi_n|^{2^*}d\vol_{\ig_{S^m}},
\end{\equ}
and
\begin{\equ}\label{c4}
	\int_{B_R^0}|\phi_n|^{p_n}d\vol_{\ig_n}
	=R_n^{-\frac{m-1}2(2^*-p_n)}
	\int_{B_{R_nR}(a_n)}|\psi_n|^{p_n}d\vol_{\ig_{S^m}}.
\end{\equ}
Moreover, since $\{\psi_n\}$ is bounded in $E$, we have
\begin{\equ}\label{c5}
	\sup_{n\geq1}\int_{B_R^0}|\phi_n|^{2^*}d\vol_{\ig_n}\leq
	\sup_{n\geq1}\int_{S^m}|\psi_n|^{2^*}d\vol_{\ig_{S^m}}
	<+\infty
\end{\equ}
for any $R>0$.

\begin{Lem}\label{concentration parameter}
	There is $\bar\lm>0$ such that	$\displaystyle	\bar\lm\leq\liminf_{n\to\infty}R_n^{\frac{m-1}2(2^*-p_n)}	\leq\limsup_{n\to\infty}R_n^{\frac{m-1}2(2^*-p_n)}\leq1$.
\end{Lem}
\begin{proof}
It follows from \eqref{choice of an}, \eqref{c4} and H\"older inequality that
	\[
	\bar\de=\int_{B_{R_n}(a_n)}|\psi_n|^{p_n}
	d\vol_{\ig_{S^m}}
	\leq\Big( \int_{B_1^0}|\phi_n|^{2^*}d\vol_{\ig_n}
	\Big)^{\frac{p_n}{2^*}}\Big( \int_{B_1^0}d\vol_{\ig_n}
	\Big)^{\frac{2^*-p_n}{2^*}}
	R_n^{\frac{m-1}2(2^*-p_n)}.
	\]
	Noting that $\ig_n$ converges to the Euclidean metric
	in the $C^\infty$-topology on $B_1^0$, we can conclude
	immediately from $p_n\nearrow2^*$ and \eqref{c5} that
	$
	\bar\de\leq C \cdot R_n^{\frac{m-1}2(2^*-p_n)}
	$
	for some constant $C>0$.
	
	On the other hand, suppose there exists some $\de>0$
	such that $R_n^{\frac{m-1}2(2^*-p_n)}\geq 1+\de$
	for all large $n$. Then, we must have
	$	\ln R_n\geq \frac{2\ln(1+\de)}{(m-1)(2^*-p_n)}\to+\infty$ as  $n\to\infty$.	This implies $R_n\to+\infty$ which is absurd.
\end{proof}

Let $\phi_n$ be as above and define $\bar L_n=D_{\ig_n}\phi_n-R_n^{\frac{m-1}2(2^*-p_n)}
H\circ\mu_n(\cdot)|\phi_n|^{p_n-2}\phi_n$. By using $\cl_{p_n}'(\psi_n)\to0$ in $E^*$ and the conformal changes of the Dirac operator, it is not difficult to check that  $\bar L_n\to0$ in $H^{-\frac12}_{\loc}(\R^m,\mbs_m)$.

In what follows, by Lemma \ref{concentration parameter},
we may assume that (up to a subsequence if necessary) $R_n^{\frac{m-1}2(2^*-p_n)}\to
\lm\in[\bar\lm,1]$ as $n\to\infty$. Since $\{\phi_n\}$ is bounded in $H_{\loc}^{\frac12}
(\R^m,\mbs_m)$, we can assume that
$\phi_n\rightharpoonup\phi_0$ in $H_{\loc}^{\frac12}
(\R^m,\mbs_m)$. And, by \eqref{c5}, it is easy to see that $\phi_0\in L^{2^*}(\R^m,\mbs_m)$ and
satisfies the equation
\begin{\equ}\label{limit equ}
D_{\ig_{\R^m}}\phi_0=\lm H(a)|\phi_0|^{2^*-2}\phi_0
\quad \text{on } \R^m.
\end{\equ}

At this point, we may apply the local analysis in \cite[Lemma 5.5]{Isobe-JFA} to see that $\phi_n\to\phi_0$ in $H_{\loc}^{\frac12}(\R^m,\mbs_m)$
as $n\to\infty$, where proof is based on the fact that $\ig_n$ converges to $\ig_{\R^m}$ in $C^\infty$-topology on bounded domains in $\R^m$. Then, by Lemma \ref{concentration parameter}, \eqref{choice of an} and \eqref{c4}, we have
\[
\lm\int_{B_1^0}|\phi_0|^{2^*}d\vol_{\ig_{\R^m}}=\bar\de,
\]
which implies $\phi_0$ is a non-trivial solution of Eq. \eqref{limit equ}. And hence, by the energy gap estimate in \cite[Section 4]{Isobe-JFA}, we can derive via standard rescaling argument that
\begin{\equ}\label{blow-up energy}
	\int_{\R^m}|\phi_0|^{2^*}dx
	\geq\frac{1}{(\lm H(a))^{m}}\big(\frac m2\big)^m \om_m.
\end{\equ}

With these preparations in hand, we may now choose
$\eta\in C^\infty(S^m)$ be such that $\eta\equiv1$ on
$B_r(a)$ and $\supp \eta\subset B_{2r}(a)$ for some
$r>0$ (for sure $r$ should not be large in the sense
that we assume $3r<inj_{S^m}$
where $inj_{S^m}$ denotes
the injective radius) and define a spinor field $z_n\in
C^\infty(S^m,\mbs(S^m))$ by $z_n=R_n^{-\frac{m-1}2}\eta(\cdot)\ov{(\mu_n)}_*
\circ\phi_0\circ \mu_n^{-1}$.
Setting $\va_n=\psi_n-z_n$, we have

\begin{Lem}\label{z-n weakly 0}
	\begin{itemize}
	\item[$(1)$] $\va_n\rightharpoonup 0$ in $E$ as $n\to\infty$.
	
	\item[$(2)$]$\cl_{p_n}'(z_n)\to0$ and $\cl_{p_n}'(\va_n)\to0$ as
	$n\to\infty$.
	\end{itemize}
\end{Lem}
\begin{proof}
The proof of $(1)$ is accomplished by showing $|z_n|_2\to0$ as $n\to\infty$ since we already have assumed $\psi_n\rightharpoonup 0$ in $E$. To see this, we split the $L^2$-integral of $z_n$ into two parts 
\[
\Big(\int_{B_{R_nR}(a_n)}+\int_{S^m\setminus B_{R_nR}(a_n)}\Big)|z_n|^2d\vol_{\ig_{S^m}}
\]
where 
\[
\int_{B_{R_nR}(a_n)}|z_n|^2d\vol_{\ig_{S^m}}
=R_n^{-m+1}\int_{B_{R_nR}^0}|\phi_0|^2d\vol_{\mu_n^*\ig_{S^m}}
=R_n\int_{B_R^0}|\phi_0|^2d\vol_{\ig_n}\to0
\]
and
\[
\int_{S^m\setminus B_{R_nR}(a_n)}|z_n|^2d\vol_{\ig_{S^m}}
\leq C \Big( \int_{B_{3r/R_n}^0\setminus B_{R}^0}
|\phi_0|^{2^*} d\vol_{\ig_{\R^m}} \Big)^{\frac2{2^*}}
\big( (3r)^m-(R_nR)^m \big)\to0
\]
for some constant $C>0$. We omit the proof of $(2)$ here, since it is standard and is similar to that of  \cite[Lemma 5.7]{Isobe-JFA}. 
\end{proof}

\begin{proof}[Proof of Proposition \ref{blow-up prop}: The concentration behavior]
	By $(2)$ of Lemma \ref{z-n weakly 0}, we have
	\[
	\cl_{p_n}(z_n)+o_n(1)=
	\cl_{p_n}(z_n)-\frac1{p_n}\cl_{p_n}'(z_n)[z_n]=
	\frac{p_n-2}{2p_n}\big( \|z_n^+\|^2-\|z_n^-\|^2 \big)
	\]
	and
	\[
	\cl_{p_n}(\va_n)+o_n(1)=
	\cl_{p_n}(\va_n)-\frac1{p_n}\cl_{p_n}'(\va_n)[\va_n]=
	\frac{p_n-2}{2p_n}\big( \|\va_n^+\|^2-\|\va_n^-\|^2 \big).
	\]
	We claim that
	\begin{claim}\label{claim1}
		$\cl_{p_n}(\psi_n)=\cl_{p_n}(z_n)+\cl_{p_n}(\va_n)
		+o_n(1)$ as $n\to\infty$.
	\end{claim}
	\noindent
	Assuming Claim \ref{claim1} for the moment, then we are going to show that $\cl_{p_n}(\va_n)\to0$ as $n\to\infty$. Indeed,
	suppose to the contrary that (up to a subsequence)
	$\cl_{p_n}(\va_n)\geq c>0$, it follows from the
	boundedness of $\{\va_n\}$ in $E$, Lemma \ref{key2} and \ref{tau p properties} that
	\begin{\equ}\label{blow-up-cl1}
	\tau_{2^*}=\tau_{p_n}+o_n(1)\leq \max_{t>0}I_{p_n}(t\va_n^+)+o_n(1)\leq \cl_{p_n}(\va_n) + o_n(1).
	\end{\equ}
	On the other hand,  we have
	\begin{eqnarray*}
		\cl_{p_n}(z_n)&=&\cl_{p_n}(z_n)-\frac12\cl_{p_n}'(z_n)[z_n]=\frac{p_n-2}{2p_n}
		\int_{S^m}H|z_n|^{p_n}d\vol_{\ig_{S^m}}+o_n(1)
		\\[0.4em]
		&=&\frac{p_n-2}{2p_n} R_n^{\frac{m-1}2(2^*-p_n)}
		\int_{B_R^0}(H\circ\mu_n)|\phi_0|^{p_n} d\vol_{\ig_n}
		+o_n(1)+o_R(1)\\[0.4em]
		&=&\frac1{2m} \lm H(a)\int_{B_R^0}|\phi_0|^{2^*}
		d\vol_{\ig_{\R^m}} + o_n(1) + o_R(1)
	\end{eqnarray*}
for $R>0$ large. Thus, by \eqref{tau value},
\eqref{blow-up energy},
$\lm\leq1$ and $H(a)\leq H_{max}$, we obtain
\begin{\equ}\label{blow-up-cl2}
	\cl_{p_n}(z_n)\geq \frac1{2m(\lm H(a))^{m-1}}
	\big(\frac m2\big)^m\om_m
	+o_n(1)\geq\tau_{2^*}+o_n(1).
\end{\equ}
Combining Claim \ref{claim1}, \eqref{blow-up-cl1} and
\eqref{blow-up-cl2}, we obtain $\cl_{p_n}(\psi_n)
\geq2\tau_{2^*}+o_n(1)$
which contradicts to \eqref{key-assumption}. Therefore, we have $\cl_{p_n}(\va_n)\to0$ as $n\to\infty$
and this, together with $\cl_{p_n}'(\va_n)\to0$, implies
$\va_n\to0$ in $E$ as $n\to\infty$ (follow the proof of \cite[Lemma 4.1]{BX}). Moreover, we can
get a lower bound for $\lm$, that is $\lm>2^{-\frac1{m-1}}$, since $H(a)\leq H_{max}$
and $\cl_{p_n}(\psi_n)<2\tau_{2^*}$.

\medskip

Now it remains to prove Claim \ref{claim1}. We point out here that this is equivalent to show
\begin{\equ}\label{claim1-1}
	\int_{S^m}(D\psi_n,\psi_n)d\vol_{\ig_{S^m}}=
	\int_{S^m}(Dz_n,z_n)d\vol_{\ig_{S^m}}+
	\int_{S^m}(D\va_n,\va_n)d\vol_{\ig_{S^m}}+o_n(1).
\end{\equ}
And since $\va_n=\psi_n-z_n$, it suffices to prove
$\int_{S^m}(Dz_n,\va_n)d\vol_{\ig_{S^m}}=o_n(1)$
as $n\to\infty$. In fact, for arbitrary $R>0$, we have
\[
\aligned
\int_{S^m}(Dz_n,\va_n)d\vol_{\ig_{S^m}}
&=\Big(\int_{B_{R_nR}(a_n)}
+\int_{B_{3r}(a_n)\setminus B_{R_nR}(a_n)}\Big)
(Dz_n,\va_n)d\vol_{\ig_{S^m}}  \\[0.4em]
&=\Big(\int_{B_R^0}+\int_{B_{3r/R_n}^0\setminus B_R^0}\Big)
(D_{\ig_n}\phi_0,\phi_n-\phi_0) d\vol_{\ig_n},
\endaligned
\]
where the first integral goes to $0$ as $n\to\infty$ since $\phi_n\to\phi_0$ in $H_{\loc}^{\frac12}(\R^m,\mbs_m)$. Meanwhile to estimate the second integral,
we first observe that (through the conformal transformation)
\[
\sup_{n}\int_{B_{3r/R_n}^0}|\phi_n-\phi_0|^{2^*}
d\vol_{\ig_{\R^m}}\leq \sup_n\int_{B_{3r}(a_n)}
|\psi_n-z_n|^{2^*}d\vol_{\ig_{S^m}}<+\infty
\]
Thus, by $d\vol_{\ig_n}\leq
C d\vol_{\ig_{\R^m}}$, we have
\[
	\Big|\int_{B_{3r/R_n}^0\setminus B_R^0}
	(D_{\ig_n}\phi_0,\phi_n-\phi_0) d\vol_{\ig_n}\Big|
	\leq C\Big( \int_{B_{3r/R_n}^0\setminus B_R^0}
	|\nabla \phi_0|^{\frac{2m}{m+1}}d\vol_{\ig_{\R^m}}
	\Big)^{\frac{m+1}{2m}}\to0
\]
as $R\to\infty$. Therefore we obtain \eqref{claim1-1}
is valid.
\end{proof}

In the above, we proved the concentration behavior $\psi_n=R_n^{-\frac{m-1}2}\eta(\cdot)\ov{(\mu_n)}_*
\circ\phi_0\circ \mu_n^{-1}+o_n(1)$ in $E$ and  the lower estimate for $\cl_{p_n}(\psi_n)$. It remains to locate the blow-up point $a\in S^m$ and to evaluate the parameter $\lm\in[\bar\lm,1]$ that are involved in the analysis.

For arbitrary $\xi\in S^m$, we embed
$S^m$ into $\R^{m+1}$ in the way that its antipodal point $-\xi$ is the North pole. Denoting $\cs_\xi$
the stereographic projection from $S^m\setminus\{-\xi\}$
to $\R^m$, we have $\cs_\xi(\xi)=0$. Particularly,
$S^m\setminus\{-\xi\}$ and $\R^m$ are conformally
equivalent, where $(\cs_\xi^{-1})^*\ig_{S^m}=f^2\ig_{\R^m}$ with
$f(x)=\frac2{1+|x|^2}$.

Let $\iota:\mbs\big(\R^m, (\cs_\xi^{-1})^*\ig_{S^m}\big)\to
\mbs(\R^m,\ig_{\R^m})$ denote the isomorphism of spinor bundles induced by Proposition \ref{conformal formula}. If $\psi\in E$ is a
solution to Eq. \eqref{nld-p} for some $p\in(2,2^*]$, then
$\phi:=\iota(f^{\frac{m-1}2}\psi\circ\cs_\xi^{-1})$ will
satisfy the transformed equation
$D_{\ig_{\R^m}}\phi=f^{\frac{m-1}2(2^*-p)}
(H\circ \cs_{\xi}^{-1})|\phi|^{p-2}\phi$ on $(\R^m,\ig_{\R^m})$.

Let $\psi_n$ be as before, and now we require further that $\cl_{p_n}'(\psi_n)\equiv0$. By the regularity results proved in \cite{Ammann}, we have $\psi_n$ is of $C^{1,\al}$ for some $\al\in(0,1)$
and are classical solutions to \eqref{equs-n}. The proof of Proposition \ref{blow-up prop} will be accomplished by collecting following two results.

\begin{Lem}\label{blow-up prop2}
 Let $a\in S^m$ be the associate blow-up point. Then $\nabla H(a)=0$.
\end{Lem}
\begin{Lem}\label{blow-up cor}
	Let $\lm\in[\bar\lm,1]$ be the parameter involved in \eqref{limit equ}. Then $\lm=1$.
\end{Lem}

\begin{proof}[Proof of Lemma \ref{blow-up prop2}]
	Let us consider the stereographic projection
	$\cs_a: S^m\setminus\{-a\}\to\R^m$ and the
	associated bundle isomorphism
	$\iota:\mbs\big(\R^m, (\cs_a^{-1})^*\ig_{S^m}\big)\to
	\mbs(\R^m,\ig_{\R^m})$. Denoted by
	$\tilde\phi_n= \iota(f^{\frac{m-1}2}\psi_n\circ \cs_a^{-1})$,
	we have that $\tilde\phi_n$ satisfies
	\begin{\equ}\label{tilde phi-n}
		D_{\ig_{\R^m}}\tilde\phi_n=H_n(x)|\tilde\phi_n|^{p_n-2}\tilde\phi_n
		\quad \text{on } (\R^m,\ig_{\R^m}).
	\end{\equ}
where, for ease of notation, we denote $H_n(x):=f(x)^{\frac{m-1}2(2^*-p_n)}
(H\circ \cs_a^{-1})(x)$.

Take $\bt\in C_c^\infty(S^m)$ be a cut-off function on
$S^m$ such that $\bt\equiv1$ on $B_{2r}(a)$ and
$\supp\bt\subset B_{3r}(a)$ where $r>0$ is the fixed radius in the definition of $z_n$. Then we are allowed
to multiply \eqref{tilde phi-n} by
$\phi_{n,k}=\pa_k\big((\bt\circ \cs_a^{-1})\tilde\phi_n \big)$
as a test spinor for each $k=1,2,\dots,m$, and
consequently 
\begin{\equ}\label{identity}
	\real\int_{\R^m}(D_{\ig_{\R^m}}\tilde\phi_n,
	\phi_{n,k}  )
	d\vol_{\ig_{\R^m}} 
=\real\int_{\R^m}
H_n	|\tilde\phi_n|^{p_n-2}(\tilde\phi_n,  \,
	\phi_{n,k}  )
	d\vol_{\ig_{\R^m}}.
\end{\equ}

Note that $(\bt\circ \cs_a^{-1})\tilde\phi_n$ has a
compact support, we may integrate by parts to get
\begin{eqnarray}\label{I1}
	0&=&\real\int_{\R^m}\pa_k\big(
	D_{\ig_{\R^m}}\tilde\phi_n,
	(\bt\circ\cs_a^{-1})\tilde\phi_n \big) d\vol_{\ig_{\R^m}}
	\nonumber \\
	&=&2\,\real\int_{\R^m}(D_{\ig_{\R^m}}\tilde\phi_n,
	\,\phi_{n,k} )
	d\vol_{\ig_{\R^m}}  
	+\real\int_{\R^m}\big(\pa_k\tilde\phi_n,
	\nabla(\bt\circ\cs_a^{-1})\cdot_{\ig_{\R^m}}\tilde\phi_n
	\big) d\vol_{\ig_{\R^m}}  \nonumber\\
	& &\quad -\,\real\int_{\R^m}\big( D_{\ig_{\R^m}}\tilde\phi_n,  \,
	\pa_k(\bt\circ\cs_a^{-1})\tilde\phi_n \big)
	d\vol_{\ig_{\R^m}},
\end{eqnarray}
where $\cdot_{\ig_{\R^m}}$ denotes the Clifford
multiplication with respect to $\ig_{\R^m}$.
Now let us evaluate the last two integrals of the
previous equality. First of all we see
from the conformal transformation and the regularity
results (see \cite{Ammann}) that
$\{\nabla\tilde\phi_n\}$ is uniformly bounded in $L^{\frac{2m}{m+1}}(\R^m,\mbs_m)$.
Hence, by the concentration behavior of $\psi_n$, we derive
\[
\Big| \int_{\R^m}\big(\pa_k\tilde\phi_n,
\nabla(\bt\circ\cs_a^{-1})\cdot_{\ig_{\R^m}}\tilde\phi_n
\big) d\vol_{\ig_{\R^m}} \Big|
\leq
C\Big( \int_{B_{3r}(a)\setminus B_{2r}(a)}
|\psi_n|^{2^*}
d\vol_{\ig_{S^m}} \Big)^{\frac1{2^*}}  \to0
\]
as $n\to\infty$. And analogously, we have
$
\big| \int_{\R^m}\big( D_{\ig_{\R^m}}\tilde\phi_n,  \,
\pa_k(\bt\circ\cs_a^{-1})\tilde\phi_n \big)
d\vol_{\ig_{\R^m}} \big|\to0
$
as $n\to\infty$. Thus, we conclude from \eqref{I1}
that
\begin{\equ}\label{I2}
	\real\int_{\R^m}\big(D_{\ig_{\R^m}}\tilde\phi_n,
	\,\phi_{n,k} \big)
	d\vol_{\ig_{\R^m}}=o_n(1) \quad \text{as } n\to\infty.
\end{\equ}

On the other hand, to evaluate the integral in the right side of \eqref{identity}, we first observe that
\begin{eqnarray}\label{I3}
	0&=&\int_{\R^m}\pa_k\big[ H_n\cdot (\bt\circ\cs_a^{-1})|\tilde\phi_n|^{p_n}
	\big]d\vol_{\ig_{\R^m}}  \nonumber\\
	&=&\int_{\R^m}
	\pa_kH_n\cdot(\bt\circ\cs_a^{-1})|\tilde\phi_n|^{p_n}
	d\vol_{\ig_{\R^m}} +p_n\real\int_{\R^m}
	H_n
	|\tilde\phi_n|^{p_n-2}(\tilde\phi_n,  \phi_{n,k})
	d\vol_{\ig_{\R^m}}   \nonumber\\
	& &\quad -\,(p_n-1)\int_{\R^m}
	H_n\cdot\pa_k
	(\bt\circ\cs_a^{-1})|\tilde\phi_n|^{p_n}
	d\vol_{\ig_{\R^m}}.
\end{eqnarray}
It is evident that the last integral converges to $0$ as
$n\to\infty$, since $\pa_k
(\bt\circ\cs_a^{-1})$ vanishes on $B_{2r}^0$. We only need to estimate the first integral. Notice that $f(x)=\frac2{1+|x|^2}$
and $\bt\circ\cs_a^{-1}$ has a compact
support on $\R^m$, we see that $H_n$, $H_n^{-1}$ and $\nabla H_n$ are bounded uniformly on $\supp(\bt\circ\cs_a^{-1})$. Hence, due to the concentration behavior of $\tilde\phi_n$, we obtain
\[
\int_{\R^m}
\pa_kH_n\cdot(\bt\circ\cs_a^{-1})|\tilde\phi_n|^{p_n}
d\vol_{\ig_{\R^m}} = \pa_kH_n(0)\int_{\R^m}
|\tilde\phi_n|^{p_n}
d\vol_{\ig_{\R^m}} +o_n(1)
\]
as $n\to\infty$. This and \eqref{I3} imply
\begin{\equ}\label{I4}
	\real\int_{\R^m}
	H_n
	|\tilde\phi_n|^{p_n-2}(\tilde\phi_n,  \phi_{n,k})
	d\vol_{\ig_{\R^m}}=-\frac1{p_n}\pa_kH_n(0)\int_{\R^m}
	|\tilde\phi_n|^{p_n}
	d\vol_{\ig_{\R^m}} +o_n(1).
\end{\equ}

Combining \eqref{identity}, \eqref{I2} and \eqref{I4},
we conclude that
\begin{\equ}\label{identity2}
	\pa_kH_n(0)\int_{\R^m}
	|\tilde\phi_n|^{p_n} d\vol_{\ig_{\R^m}}= o_n(1)
\end{\equ}
as $n\to\infty$. Since we already know from the blow-up analysis that
\[
\lim_{n\to\infty}\int_{\R^m}f^{\frac{m-1}2(2^*-p_n)}
|\tilde\phi_n|^{p_n}d\vol_{\ig_{\R^m}}
=\lim_{n\to\infty}\int_{S^m}
|\psi_n|^{p_n}d\vol_{\ig_{S^m}}>0,
\]
\eqref{identity2} tells us nothing but
$\pa_kH_n(0)\equiv0$. Notice that $\nabla f(0)=0$ and
$k$ can be varying from $1$ to $m$, we have
$\nabla(H\circ\cs_a^{-1})(0)=0$, i.e. $\nabla H(a)=0$
which completes the proof.
\end{proof}

\begin{proof}[Proof of Lemma \ref{blow-up cor}]
	The proof is very similar to that of Lemma \ref{blow-up prop2}. Let us recall the equation under stereographic projection \eqref{tilde phi-n} and consider the conformal transformation of $\tilde\phi_n$ by $
	\tilde\phi_{n,R}(x)=R^{\frac{m-1}2}\tilde\phi_n(Rx)$  for $R>0$. Then we have
	\begin{\equ}\label{eq-tilde-n-r}
		D_{\ig_{\R^m}}\tilde\phi_{n,R}=R^{\frac{m-1}2(2^*-p_n)}\widehat H_{n,R}|\tilde\phi_{n,R}|^{p_n-2}\tilde\phi_{n,R} \quad \text{on } \R^m
	\end{\equ}
	where, for ease of notations, we have denoted $\widehat H_{n,R}(x)=f^{\frac{m-1}2(2^*-p_n)}(Rx)\cdot
	H\circ \cs_a^{-1}(Rx)$. 
	
	Let $\bt\in C_c^\infty(S^m)$ be the same cut-off function as in \eqref{identity}, and set $
	\hat\phi_{n,R}(x)=\bt\circ\cs_a^{-1}(Rx)\cdot \tilde\phi_{n,R}(x)$. Then we can multiply \eqref{eq-tilde-n-r} by $\hat\phi_n^*=\frac{\pa}{\pa R}\big|_{R=R_n}\hat\phi_{n,R}$ (where $R_n>0$ was determined in \eqref{choice of an}) to get
	\[
		\real\int_{\R^m}(D_{\ig_{\R^m}}\tilde\phi_{n,R_n},\hat\phi_n^*)d\vol_{\ig_{\R^m}}=R_n^{\frac{m-1}2(2^*-p_n)}\real\int_{\R^m}\widehat H_{n,R_n}|\tilde\phi_{n,R_n}|^{p_n-2}(\tilde\phi_{n,R_n},\hat\phi_n^*)d\vol_{\ig_{\R^m}}.
	\]
	In what follows, we are going to estimate the above two integrals.
	
	Via the conformal transformation property of the Dirac operator, we see that
	\begin{\equ}\label{id1}
		\int_{\R^m}(D_{\ig_{\R^m}}\tilde\phi_{n,R},\hat\phi_{n,R})d\vol_{\ig_{\R^m}}
		\equiv\int_{S^m}(D\psi_n,\bt\psi_n) d\vol_{\ig_{S^m}}
	\end{\equ}
	and
	\begin{\equ}\label{id2}
		\int_{\R^m}(\bt\circ\cs_a^{-1})(Rx)\cdot\widehat H_{n,R}|\tilde\phi_{n,R}|^{p_n}d\vol_{\ig_{\R^m}}\equiv R^{\frac{m-1}2(p_n-2^*)}
		\int_{S^m}\bt H|\psi_n|^{p_n}d\vol_{\ig_{S^m}}
	\end{\equ}
for all $R>0$.	Hence, by taking derivative with respect to $R$ in \eqref{id1} and \eqref{id2}, and estimate similar as in the proof of Lemma \ref{blow-up prop2}, we can obtain
	\[
	(2^*-p_n)R_n^{-1}\int_{S^m}\bt H|\psi_n|^{p_n}d\vol_{\ig_{S^m}}=O_n(1) \quad \text{as } n\to\infty.
	\]
	Since the concentration behavior suggests that $\lim_{n\to\infty}\int_{S^m}\bt H|\psi_n|^{p_n}d\vol_{\ig_{S^m}}>0$, we find
	$2^*-p_n=O(R_n)$ as $n\to\infty$. Therefore we derive
	\[
	\lm=\lim_{n\to\infty}R_n^{\frac{m-1}2(2^*-p_n)}=\lim_{n\to\infty}e^{O(1)R_n\ln R_n}=1.
	\]
	This completes the proof.
\end{proof}

\vspace{2mm}
{\sc Tian Xu\\
 Center for Applied Mathematics, Tianjin University\\
 Tianjin, 300072, China}\\
 xutian@amss.ac.cn

\end{document}